\documentclass[12pt]{article}

\usepackage{amsmath,amssymb,amsthm}
\usepackage{amsfonts}
\usepackage{color}
\usepackage{bm}
\usepackage{algorithm}
\usepackage{algpseudocode}
\usepackage{graphicx}
\usepackage[left=2.6cm,right=2.2cm,top=2.5cm,bottom=2.5cm]{geometry}

\newcommand{\be}{\begin{eqnarray}}
\newcommand{\e}{\end{eqnarray}}
\newcommand{\bes}{\begin{eqnarray*}}
\newcommand{\es}{\end{eqnarray*}}
\newcommand{\beq}{\begin{equation}}
\newcommand{\eeq}{\end{equation}}

\newcommand{\eps}{\varepsilon}

\newcommand{\pp}{{\boldsymbol\lambda}}

\newcommand{\uu}{{\bf u}}
\newcommand{\vv}{{\bf v}}

\usepackage{tikz} \usetikzlibrary{positioning,shapes,calc,backgrounds,fit,plotmarks}
\usepackage{pgfplots}
\pgfplotsset{compat=newest,
every axis/.append style={axis x line=bottom,
			  scale only axis,
			  },
}
\tikzset{dashdot/.style={dash pattern=on .4pt off 3pt on 4pt off 3pt}}
\pgfplotscreateplotcyclelist{stokescontrol}{
         black, mark=o, draw, line width=1pt\\%
         blue, dotted, line width=1.6pt, mark=x\\%
         red, dashed, line width=2.4pt, mark=+\\%
         green!50!black, opacity=0.4, line width=3.0pt, mark=triangle*\\%
         orange, dashdot, line width=3.0pt, mark=*\\%
}

\newtheorem{theorem}{Theorem}

\date{March 17, 2017}

\begin{document}
\title{Solving optimal control problems governed by random Navier-Stokes equations using low-rank methods} 
\author{Peter Benner\footnotemark[1], Sergey Dolgov\footnotemark[2], Akwum Onwunta\footnotemark[1], Martin Stoll\footnotemark[1]}
\renewcommand{\thefootnote}{\fnsymbol{footnote}}
\footnotetext[2]{University of Bath,
Department of Mathematical Sciences,
Claverton Down,
BA2 7AY Bath, United Kingdom (\texttt{S.Dolgov@bath.ac.uk}). S. Dolgov gratefully acknowledges funding by the EPSRC Fellowship EP/M019004/1.}
\footnotetext[1]{Computational Methods in Systems and Control Theory, Max Planck Institute
for Dynamics of Complex Technical Systems,
Sandtorstr. 1,
39106 Magdeburg,
Germany. \textit{Email addresses:} \texttt{\{benner,onwunta,stollm\}@mpi-magdeburg.mpg.de}}

\renewcommand{\thefootnote}{\arabic{footnote}}

\maketitle

\begin{abstract}
Many  problems in computational science and engineering are simultaneously characterized by
the following  challenging issues:
uncertainty, nonlinearity, nonstationarity and high dimensionality.
Existing numerical  techniques for  such models
would typically require considerable computational and storage resources.
This is the case, for instance,
for an optimization  problem governed by  time-dependent Navier-Stokes equations 
with uncertain inputs. In particular, the stochastic Galerkin finite element
method often leads to a prohibitively high dimensional saddle-point
system with tensor product structure.  In this paper, we approximate the solution by the low-rank Tensor Train decomposition,
and present a numerically efficient algorithm to solve the optimality equations directly in the low-rank representation.
We show that the solution of the vorticity minimization problem with a distributed control admits a representation with ranks 
that depend modestly on model and discretization parameters even for high Reynolds numbers.
For lower Reynolds numbers this is also the case for a boundary control.
This opens the way for a reduced-order modeling of the stochastic optimal flow control with a moderate cost at all stages.
\end{abstract}

{\textit{Keywords: }
Stochastic Galerkin system, iterative methods, PDE-constrained optimization, saddle-point system, tensor train format, low-rank solution,  preconditioning, Schur complement. \\
}

\section{Introduction}
We consider the numerical simulation of
optimization problems constrained by  partial differential equations (PDEs). 
This class of problems can be computationally challenging.
This is particularly so if the constraints are time-dependent PDEs since time-stepping methods quickly reach their
limitations due to the enormous demand for storage \cite{StoB13}.
The computational complexity associated with these problems is further increased when the constraints are nonlinear  \cite{ChoLu2015} 
and contain (up to countably many) parametric or uncertain inputs \cite{BSOS215,BOS15a}.
This is the core problem of this paper.
A viable solution approach to optimization problems with stochastic constraints  employs the spectral stochastic Galerkin 
finite element method (SGFEM). However, this intrusive approach\footnote{Generally speaking, SGFEM techniques are intrusive in the
sense that the codes for the associated deterministic problems
cannot be directly reused. These methods are mainly non-ensemble-based methods and require the solution of discrete systems that couple all spatial and 
probabilistic degrees of freedom.} leads  to the so-called \textit{curse of dimensionality}, in the sense
that it results in prohibitively high dimensional linear systems with tensor product structure \cite{BSOS215,BOS14,EW12}.

It is worth pursuing  computationally efficient ways to simulate optimization problems with stochastic constraints 
using SGFEMs since the Galerkin approximation yields the best approximation
with respect to the energy norm, as well as a favorable framework for error estimation \cite{BPS13}.
In order to cope with the curse of dimensionality we exploit  the underlying mathematical structure of the discretized optimality system.
We develop a low-rank technique based on recent advances in numerical
tensor methods  \cite{hackbusch-2012,bokh-surv-2014} for efficient solution of an optimization  problem governed by nonlinear PDEs
with random coefficients. More specifically, we numerically  simulate an  \underline{O}ptimization  \underline{P}roblem constrained
by  time-dependent \underline{N}avier-\underline{S}tokes equations
with random coefficients\footnote{We remark here that our approach can be easily generalized 
to other nonlinear optimization problems with stochastic constraints as well.} (\underline{OPNS}).
Our aim in this paper is to lift the curse of dimensionality inherent in the
OPNS and allow for efficient simulations of the  model on not much more than an average desktop computer.
Such simulations would enhance
the understanding of the underlying physical model as the computed data can then be used for the quantification of the statistics
of the system response. 

Alternative  approaches to tackle optimization problems with stochastic constraints  include  stochastic collocation schemes \cite{BW09, KHRW13, TKXP12}, 
as well as Monte Carlo methods \cite{AHU16}. These methods are essentially sampling-based and non-intrusive.  
However, for optimization problems, the SGFEM exhibits superior performance compared to the stochastic collocation method \cite{EW12}.
This is due to the fact that, unlike SGFEM,  the non-intrusivity property of the stochastic collocation method
is lost when moments of the state variable appear in the cost functional, or when the control is a
deterministic function. On the other hand, Monte Carlo methods are relatively straightforward to implement.
However, they generally converge rather slowly and do not exploit the  regularity with respect to the parameters that the
solution might have \cite{TKXP12}.

The rest of the paper is organised as  follows. First, we present in Section \ref{probst} the PDE-constrained
optimization that we would like to solve, as well as the necessary mathematical concepts and notation
on which we shall rely in the rest of our discussion. Next, we proceed to Section \ref{discrete_prob} to discuss
the SGFEM discretization of the problem. Section \ref{tt_decom} presents our low-rank iterative
solver and preconditioner which we use to tackle the high-dimensional saddle point systems arising from
the SGFEM discretization of the optimization problem. Finally, in Section \ref{experiments}, we present
numerical experiments to illustrate the performance of the low-rank approach.

\section{Problem statement and mathematical description}
\label{probst}
Let $(\Omega,\mathcal{F},{\mathbb{P}})$ be a complete probability space, where
 $\Omega$ is the set of outcomes, $\mathcal{F}\subseteq 2^{\Omega}$ is the $\sigma$-algebra of events, and $\mathbb{P}:\mathcal{F}\rightarrow [0,1]$ is an appropriate probability measure.
Let $\mathcal{D}\subset \mathbb{R}^d$ with $d\in\{1,2,3\},$ be a bounded spatial domain with a 
 piecewise Lipschitz boundary $\Upsilon.$
Moreover, we consider a finite time interval $[0,T]$.
A random field $z: \Omega \times \mathcal{D} \rightarrow \mathbb{R},$
 means that $z(\cdot,{\bf x})$ is a random variable defined on  $(\Omega,\mathcal{F},{\mathbb{P}})$ for each  ${\bf x}\in \mathcal{D}.$
We assume that $z$ belongs to the  tensor product Hilbert space $L^2(\Omega) \otimes L^2(\mathcal{D})$ which is endowed with the norm
\[
||z||_{L^2(\Omega) \otimes L^2(\mathcal{D})}:=\left(\int_\Omega ||z(\omega,\cdot)||^2_{L^2(\mathcal{D})} \;d\mathbb{P}(\omega)\right)^{\frac{1}{2}} <\infty,
\]
where ${L}^2(\Omega):={L}^2(\Omega,\mathcal{F},{\mathbb{P}}).$
 For any random variable $g$ defined on  $(\Omega,\mathcal{F},{\mathbb{P}}),$
 the mean $\mathbb{E}(g)$ of $g$ is given by
\be
\label{mean}
 \left<g\right>:=\mathbb{E}(g)=\int_\Omega g\;d\mathbb{P}(\omega)<\infty.
\e
For a Hilbert space $H$ of functions on $\mathcal{D}$ and a time interval $[0,T],$ we write  $L^2(0,T;H)$ for  the tensor product space $L^2([0,T])\otimes H.$
However, in particular, we  write $L^2(0,T;\mathcal{D})$ for $L^2(0,T;L^2(\mathcal{D}));$ we will be using these last two notations interchangeably.


We consider an optimization problem involving an incompressible Newtonian flow  with an uncertain inflow in a
backward step domain.
More precisely, the OPNS we want to solve is given by the minimization of the total vorticity: 
\begin{align}
\label{J}
\mathcal{J} & = \frac{1}{2}\|\nabla\times\mathbf{v}\|^2
+\frac{\beta}{2}\|\mathbf{u}\|^2,
\end{align}
where
${\bf v},\; {\bf u}: [0,T] \times \Omega \times \mathcal{D}\rightarrow \mathbb{R}^d$ are random fields,
representing the state (velocity) and  the control functions.
The regularization constant $\beta$ in \eqref{J} balances between minimization of the vorticity and penalization of the control magnitude.
The objective function $\mathcal{J}({\bf v},{\bf u}) $ is a deterministic quantity with uncertain components ${\bf v}$  and ${\bf u}$.

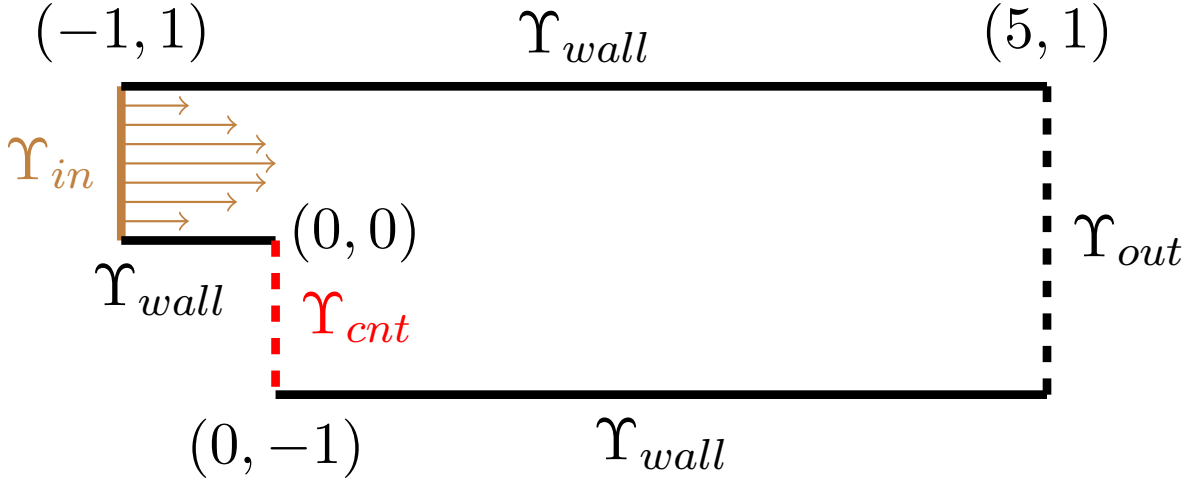
\begin{figure}[h!]
  \begin{center}
\caption{An uncertain inflow in a backward step domain}
\label{NS_flow}
\resizebox{\linewidth}{!}{
%
%
%
%
\begin{tikzpicture}{\small
 \draw[ultra thick,brown] (-1,0) -- (-1,1) node[left,midway] {$\Upsilon_{in}$}; 
 \draw[ultra thick] (-1,1) -- (5,1) node[above,midway] {$\Upsilon_{wall}$}; 
 \draw[ultra thick] (-1,0) -- (0,0) node[below,near start] {$\Upsilon_{wall}$}; 
 \draw[ultra thick,red,dashed] (0,0) -- (0,-1) node[right,midway] {$\Upsilon_{cnt}$};  
 \draw[ultra thick] (0,-1) -- (5,-1) node[below,midway] {$\Upsilon_{wall}$}; 
 \draw[ultra thick,dashed] (5,-1) -- (5,1) node[right,midway] {$\Upsilon_{out}$}; 

 \draw[->,brown] (-1,0.125) -- (-0.5625,0.125);
 \draw[->,brown] (-1,0.250) -- (-0.2500,0.250);
 \draw[->,brown] (-1,0.375) -- (-0.0625,0.375);
 \draw[->,brown] (-1,0.500) -- ( 0.0000,0.500);
 \draw[->,brown] (-1,0.625) -- (-0.0625,0.625);
 \draw[->,brown] (-1,0.750) -- (-0.2500,0.750);
 \draw[->,brown] (-1,0.875) -- (-0.5625,0.875);

\node[anchor=south] at (-1,1.02) {$(-1,1)$};
\node[anchor=south] at (5,1.02) {$(5,1)$};
\node[anchor=north] at (0,-1.02) {$(0,-1)$};
\node[anchor=west] at (-0.05,0.05) {$(0,0)$};
}
\end{tikzpicture}
}
\end{center}
\end{figure}

The minimization of $\mathcal{J}$ is considered subject, $\mathbb{P}$-almost surely, 
to the Navier-Stokes equations\footnote{In this paper, we do not consider the
case of state- or control- or mixed control-state-constrained problems \cite{HIK02,SPW11,PST15}. 
These problems can be tackled via, for instance, semi-smooth Newton algorithms \cite{HS10, IK03, Kanzow04}.}
\be%
 \label{modelprob}
\left\{
\begin{aligned}
    \frac{\partial {\bf v}}{\partial  t}  - \nu \Delta {\bf v} + ({\bf v}\cdot \nabla) {\bf v} + \nabla p &= F_d {\bf u},
       & \mbox{in} \quad  (0,T]\times \Omega  \times  \mathcal{D},  \\
    -\nabla \cdot {\bf v} &= 0, & \mbox{in} \quad (0,T]\times \Omega  \times\mathcal{D} , \\
     {\bf v} &= {\boldsymbol \theta}, & \mbox{on} \quad  (0,T]\times  \Omega \times \Upsilon_{in},\\
     {\bf v} &= 0, & \mbox{on} \quad  (0,T]\times \Omega \times \Upsilon_{wall} ,\\
     \frac{\partial \vv}{\partial  {\bf n}}  &= 0,  & \mbox{on} \quad (0,T]\times \Omega  \times  \Upsilon_{out},\\
     \frac{\partial \vv}{\partial  {\bf n}}  &= F_b {\bf u},  & \mbox{on} \quad (0,T]\times \Omega  \times  \Upsilon_{cnt},\\
     {\bf v}(0,\cdot,\cdot) &= 0, & \mbox{in} \quad \Omega   \times \mathcal{D}.
\end{aligned}
\right.
\e
Here,
${\boldsymbol \theta}:  [0,T] \times \Omega \times \Upsilon_{in} \rightarrow \mathbb{R}^d$ is the inflow boundary condition,
$p: [0,T] \times \Omega \times \mathcal{D} \rightarrow \mathbb{R}$ is the pressure, and
${\bf n}$ is the outward-pointing  normal to the boundary.
The linear  operators $F_d$ and $F_b$ apply the control to the appropriate part of the system.
We consider two cases.
\begin{enumerate}
 \item \emph{Distributed} (or full) control: $F_d = I$, $F_b = 0$, and
 \item \emph{Boundary} control: $F_d=0$, and $F_b = I_{\Upsilon_{cnt}}$, the identity on $\Upsilon_{cnt}$ and zero otherwise.
\end{enumerate}

The  parameter  $\nu$ is the  kinematic  viscosity.
There exist variations of this problem such as a deterministic control \cite{EW12} or an uncertain domain \cite{BH14}.
 
In our experiments, we use a zero initial condition, applying instead an exponentially growing stochastic inflow
\[
\boldsymbol\theta = \begin{bmatrix}\kappa({\bf x},\omega)(1-\mathrm{e}^{-t}) \\ 0\end{bmatrix}
\]
 (see Fig. \ref{NS_flow}), where $\kappa({\bf x},\omega)$ is a random field on $\Upsilon_{in}$.
We assume that $\kappa$ admits the  Karhunen-Loeve expansion (KLE):
\begin{equation}
\kappa({\bf x},\omega) = 4 x_2 (1-x_2) + \sum_{k=1}^{\infty} \frac{1}{2} \gamma^{-k} \cdot \sin(2 \pi k x_2) \cdot \xi_k(\omega).
\label{eq:kle_in}
\end{equation}
Here, the first term is the mean parabolic inflow, whereas $\xi_k$, $k=1,2,\ldots$ are independent uniformly distributed random quantities, $\xi_k \sim \mathcal{U}(-1,1)$.
The parameter $\gamma$ governs the KLE decay rate, and hence the smoothness of the field.
In computational practice, we truncate (\ref{eq:kle_in}) after $m\in \mathbb{N}$ terms such that the error is
sufficiently small:
\begin{equation}
\kappa({\bf x},\omega) \approx \kappa_m({\bf x},\xi(\omega))=  4 x_2 (1-x_2) + \sum_{k=1}^{m} \frac{1}{2} \gamma^{-k} \cdot \sin(2 \pi k x_2) \cdot \xi_k(\omega),
\label{eq:kle_in_new}
\end{equation}
where $\xi:=\{\xi_1,\xi_2,\ldots,\xi_m\}.$
This coefficient mimics the (truncated) Fourier expansion of the Matern covariance function for $\kappa.$
The assumption that  $\kappa(x,\omega)$ admits a KLE representation allows to transform
the stochastic OPNS into a parametric deterministic problem, depending on $\xi:=\{\xi_1,\xi_2,\ldots,\xi_m\}.$
For brevity, since we consider $m$ as an a priory model parameter, we always denote $\kappa_m$ by $\kappa$ in the rest of the paper.

We also assume that each random variable is  characterized by a probability density function $\rho_i:\Gamma_i\rightarrow[0,1].$
If the distribution measure of the random vector $\xi(\omega)$ is absolutely continuous with respect
to the Lebesgue measure, then there exists a joint probability density function
$\rho:\Gamma\rightarrow \mathbb{R}^+,$ where $\Gamma:=\bigotimes^m_{k=1}\Gamma_k\subset \mathbb{R}^m,  \;\;  \rho(\xi)=\prod^m_{k=1}\rho_k(\xi_k), $  and $\rho\in L^{\infty}(\Gamma).$
In particular, given the parametric representation (\ref{eq:kle_in_new}) of  $\kappa({\bf x},\omega),$
 the Doob-Dynkin Lemma \cite{BTZ04}, guarantees that the  solution, for example, $\vv$,
corresponding to the the OPNS  as given by (\ref{J})  and (\ref{modelprob})  admits  exactly the same parametrization;
that is,
\be
\label{solnv}
 \vv(t,\omega,{\bf x})=\vv(t,\xi_1(\omega),\xi_2(\omega),\ldots, \xi_m(\omega),{\bf x}).
\e
Furthermore, we can now replace the probability space $(\Omega,\mathcal{F},{\mathbb{P}})$ with $(\Omega,\mathbb{B}(\Gamma),\rho(\xi)d\xi),$ where $\mathbb{B}(\Gamma)$ denotes the
Borel $\sigma$-algebra on $\Gamma$ and $\rho(\xi)d\xi$ is the finite measure of the vector $\xi.$
Besides, denoting  the space of square-integrable random variables with respect to the density $\rho$ by ${L}_\rho^2(\Gamma),$
 we introduce the space  $L_\rho^2(\Gamma) \otimes L^2(\mathcal{D})$ equipped with the norm
\be
\label{hilrho}
||z||_{L_\rho^2(\Gamma) \otimes L^2(\mathcal{D})}:=\left(\int_\Gamma ||z(\xi,\cdot)||^2_{L^2(\mathcal{D})}\rho(\xi) \;d\xi\right)^{\frac{1}{2}} <\infty.
\e
Similarly, using equation (\ref{mean}), we have
\be
\label{exp2}
 \left<g\right>:= \mathbb{E}[g]= \int_\Gamma g(\xi) \rho(\xi) \;d\xi <\infty.
\e
%

For the  numerical simulation of the  OPNS given by \eqref{J} and \eqref{modelprob}, 
we will adopt the so-called $optimize$-$then$-$discretize$ (OTD) strategy
in which case  we first build an infinite dimensional Lagrangian and then 
consider its variation with respect to state, pressure, control, and 
two Lagrange multipliers that can be identified as the adjoint state
${{\boldsymbol{\lambda}}}$ and adjoint pressure $\mu$  \cite{book::FT2010,Hinze2000}. 
To this end, observe now that the optimization problem under consideration is nonlinear due to the nonlinearity
of the convective term $({\bf v}\cdot \nabla) {\bf v}$.
Both Newton and Picard  iterations have shown
to be good iterative solvers  to tackle these nonlinear equations \cite{ESW14}. 
Since the Picard iteration has a larger radius of convergence compared to the Newton iteration,
our  choice in this contribution is the Picard iteration. We apply the
  so-called Karush-Kuhn-Tucker procedure \cite[Chapter 8.2]{ESW14} 
 to  obtain the optimality system \cite{SSNS17}:
\begin{alignat}{3}
\label{kkt_cond1}
\partial_t\vv-\nu \Delta \vv+ \left(\bar\vv \cdot \nabla\right)\vv+\bigtriangledown p &= F_d{\bf u}  &&\nonumber\\
\nabla\cdot \vv &= 0  && \nonumber\\
\mbox{on} \quad \Upsilon_{wall}: \qquad \vv &= 0  \nonumber\\\
\mbox{on} \quad \Upsilon_{in}:   \qquad {\bf v} & = {\boldsymbol \theta} \nonumber\\
\mbox{on} \quad \Upsilon_{out}: \qquad \frac{\partial \vv}{\partial  {\bf n}}  &= 0, \\
\mbox{on} \quad \Upsilon_{cnt}: \qquad  \frac{\partial \vv}{\partial  {\bf n}}  &= F_b {\bf u}, \nonumber\\
     {\bf v}(0,\cdot,\cdot) &= 0,
\nonumber&\\
\nonumber&\\
\label{kkt_cond2}
-\partial_t\pp- \nu \Delta \pp- \left(\bar\vv \cdot \nabla\right)\pp+\left(\nabla\bar\vv\right)^{T}\pp+\bigtriangledown \mu&=
-\nabla \times (\nabla \times {\bf v}) &&\nonumber\\
\nabla\cdot \pp&=0  &&\nonumber\\
\mbox{on} \quad \Upsilon_{wall}\cup \Upsilon_{in}: \qquad \pp&=0  \\
\mbox{on} \quad \Upsilon_{out} \cup \Upsilon_{cnt}: \qquad \frac{\partial\pp}{\partial\mathbf{n}} & =0 \nonumber\\
\pp(T,\cdot,\cdot) & = 0, \nonumber\\
\nonumber&\\
\label{kkt_cond3}
\beta \uu+(F_d+F_b)\pp &=0,  &&
\end{alignat}
where 
${\bar \vv}$ denotes the velocity from the previous Picard iteration.
Having solved this system, we update ${\bar\vv} = {\vv}$ and so on until convergence.

\section{The discrete problem}
\label{discrete_prob}
The stochasticity of the inflow $\kappa({\bf x},\omega)$ is inherited by all solution components.
For example, the velocity becomes a function of time, $d$ spatial variables and $m$ random parameters, $\vv(t,\xi,{\bf x})$.
This makes the problem far more challenging than its deterministic counterpart.
In this paper, we employ simultaneous independent discretization of all variables, and the stochastic Galerkin method (SGFEM) w.r.t. the parameters $\xi$.

%
The  SGFEM is an intrusive approach
 in which  one seeks the solution in a finite-dimensional subspace
 $   \mathcal{Y}_{n_\xi} \otimes  \mathcal{X}_h \subset L_\rho^2(\Gamma) \otimes L^2(0,T;L^2(\mathcal{D}))$
consisting of tensor products of deterministic functions defined on the spatial domain
and stochastic functions defined on the probability space
\cite{GWZ14,PE09}. 
Different classes of SGFEMs are distinguished by their choices for $\mathcal{Y}_{n_\xi}.$
When all random variables $\xi_k$ are independent
and identically distributed Gaussian, the basis of multidimensional
Hermite polynomials of the total degree $n_\xi$ is called the {\it polynomial chaos}, a terminology originally introduced by Norbert
Wiener \cite{NW38} in the context of turbulence modeling.
The use of Hermite polynomials ensures that the corresponding basis functions are orthogonal with
respect to the Gaussian probability measure. This leads to sparse linear systems, a crucial property that
must be exploited for fast solution schemes \cite{PE09}.
Relying on the  fact that there exists a one-to-one correspondence between the probability density functions
of alternative distributions and the weight functions of certain orthogonal polynomials,
 the concept of Hermite polynomials chaos has  been extended to {\it generalized} polynomial chaos \cite{XK03}.
For instance, if uniform
random variables (having support on a bounded interval)
are chosen, then Legendre polynomials are the correct choice. Similarly, Jacobi polynomials go with beta-distributed random variables.
When random variables with bounded images are used, the convergence and approximation properties
of the resulting SGFEM are discussed in \cite{BTZ04}.
Some SGFEM classes use tensor products of piecewise
polynomials on the subdomains $\Gamma_k\subset \Gamma$ \cite{BTZ04, EEOS05, NT09}, where
the polynomial degree is fixed and approximation is improved by refining the partition of $\Gamma.$

Throughout this paper we use the classical, so-called {\it spectral} SGFEM (see e.g. \cite{EF07, GS91,  LKDNG03, PE09}), which employs global Legendre polynomials for uniformly distributed $\xi_k$.
We discretize each $\xi_k$ independently via polynomials of the maximal degree $n_\xi-1$.
That is, we use the space
\begin{equation*}
\mathcal{Y}_{n_\xi} = \left\{\mathrm{span} \left(\psi_j(\xi) = \prod_{k=1}^{m} \psi_{j_k} (\xi_k)\right), \quad \forall j_k = 0,\ldots,n_\xi-1, \quad j=(j_1,\ldots,j_m) \right\}.
\end{equation*}
This implies that the global discretization is done via multilinear polynomials of total degree $m (n_\xi-1)$.
The total number of basis functions in $\mathcal{Y}_{n_\xi}$ is $N_\xi=n_\xi^{m}$, which can be prohibitively large for a straightforward storage.
We circumvent this problem by using low-rank product decompositions on the discrete level in Section \ref{tt_decom}.

For the spatial discretization we use  mixed finite element spaces
with stable elements, i.e. elements that satisfy the {\it inf-sup} condition, e.g. the $Q_2$-$Q_1$ pair \cite{JS14}.
This gives us different spaces $V_h=\mbox{span}\{\phi_1,\ldots, \phi_{N_v}\} \subset L^2(0,T;H^1_0(\mathcal{D}))$ 
for the velocity and $\tilde V_h=\mbox{span}\{{\tilde\varphi}_1,\ldots, {\tilde\varphi}_{N_p}\}\subset L^2(0,T;L^2(\mathcal{D}))$ for the pressure.

The time is discretized on a uniform grid, $t_n = \tau n$, $n=1,\ldots,N_t$, such that $t_{N_t}=T$.
The fully discrete coefficients of the velocity will be denoted as $\vv_h(n,j,k)$,
and similarly the coefficients of $p,{\bf u},\pp$, and $\mu$ are denoted by $p_h, {\bf u}_h, \pp_h$ and $\mu_h$, respectively.
The solution functions are approximated as
\begin{equation}
\label{sbdiscrete}
\begin{aligned}
  \vv(t_n,\xi,{\bf x}) & \approx \sum_{j=0}^{N_\xi-1} \sum_{k=1}^{N_v}\vv_h(n,j,k) \psi_j(\xi) \phi_k({\bf x}) , \\
  p(t_n,\xi,{\bf x}) & \approx \sum_{j=0}^{N_\xi-1} \sum_{k=1}^{N_p}p_h(n,j,k) \psi_j(\xi) {\tilde\varphi}_k({\bf x}).
\end{aligned}
\end{equation}
Analogous expressions hold for ${\bf u}, \pp$ and $\mu$.

We plug these expressions into \eqref{kkt_cond1}--\eqref{kkt_cond3}, and project the optimality system onto $\mathrm{span}(\psi_j \phi_k)$ for the velocity equations and $\mathrm{span}(\psi_j \tilde\varphi_k)$ for the pressure equations.
We introduce the following matrices, composed from the corresponding bilinear forms:
\begin{itemize}
 \item $L_0(k,k') = \int\limits_{\mathcal{D}} \nabla \phi_k({\bf x}) \cdot \nabla \phi_{k'}({\bf x}) d{\bf x}$ is the Laplace matrix for a single velocity component, and $L = \mathrm{blkdiag}(L_0, L_0)$ is the Laplace matrix for both velocity components.
 \item $M_0(k,k') = \int\limits_{\mathcal{D}}  \phi_k({\bf x}) \phi_{k'}({\bf x}) d{\bf x}$ and $M = \mathrm{blkdiag}(M_0,M_0)$, the mass matrix.
 \item $B(k,k') = \int\limits_{\mathcal{D}}  \tilde\varphi_k({\bf x}) \nabla \phi_{k'}({\bf x}) d{\bf x}$, the mixed gradient matrix.
 \item $N_{n} = \mathrm{blkdiag}(N_{0,n},N_{0,n})$, the convection matrix, where
 \begin{eqnarray}
 N_{0,n}(jk,j'k') & = & \int\limits_{\Gamma} \psi_j \sum_{j''=0}^{N_\xi-1} N_{s}[\bar\vv_h(n,j'',:)](k,k') \psi_{j''} \psi_{j'} \rho d\xi, \\
 N_{s}[\bar\vv_h](k,k') & = & \int\limits_{\mathcal{D}} \phi_k({\bf x}) \sum_{k''=1}^{N_v} \phi_{k''}({\bf x}) \bar\vv_h(k'') \cdot \nabla \phi_{k'}({\bf x}) d{\bf x}.
 \label{eq:N_detailed}
 \end{eqnarray}
 \item $W_{n}(jk,j'k') = \int\limits_{\Gamma} \psi_j(\xi) \sum_{j''=0}^{N_\xi-1} W_{s}[\bar\vv_h(n,j'',:)](k,k') \psi_{j''}(\xi) \psi_{j'}(\xi) \rho(\xi)d\xi$, the adjoint convection matrix, where
\begin{equation}
 W_{s}[\bar\vv_h](k,k') = \int\limits_{\mathcal{D}} \phi_k({\bf x}) \sum_{k''=1}^{N_v}\nabla\bar\vv_h(k'') \phi_{k''}({\bf x}) \phi_{k'}({\bf x}) d{\bf x}.
\label{eq:Ws}
\end{equation}
\end{itemize}
Since the Legendre polynomials are orthogonal, a proper normalization can turn the stochastic mass matrix into an identity.

We use the implicit Euler approximation for the time derivative, although it can be replaced by higher-order schemes.
Overall, this gives the following discrete optimality equations.
\allowdisplaybreaks
\begin{alignat}{3}
\label{kkt_discr1}
\frac{\vv_h(n,j,k)-\vv_h(n-1,j,k)}{\tau} - \sum_{k'=1}^{N_v} \nu L(k,k')\vv_h(n,j,k') \nonumber\\
 + \sum_{j'=0}^{N_\xi-1}\sum_{k'=1}^{N_v} N_{n}(jk,j'k') \vv_h(n,j',k') + \sum_{k'=1}^{N_p} B(k',k)^T p_h(n,j,k') &= \sum_{k'=1}^{N_v} M_{cnt}(k,k') {\bf u}_h(n,j,k')  &&\nonumber\\
\sum_{k'=1}^{N_v} B(k,k') \vv_h(n,j,k') &= 0  && \\
\mbox{for} \quad k \in \Upsilon_{wall}: \qquad \vv_h(n,j,k) &= 0  \nonumber\\\
\mbox{for} \quad k \in \Upsilon_{in}:   \qquad \vv_h(n,j,k) & = {\boldsymbol \theta}_h(n,j,k) \nonumber\\
     \vv_h(0,j,k) &= 0,
\nonumber&\\
\nonumber&\\
\label{kkt_discr2}
\frac{\pp_h(n,j,k)-\pp_h(n+1,j,k)}{\tau} - \sum_{k'=1}^{N_v} \nu L(k,k')\pp_h(n,j,k')   \nonumber\\
 + \sum_{j'=0}^{N_\xi-1}\sum_{k'=1}^{N_v} (-N_{n}(jk,j'k') + W_{n}(jk,j'k')) \pp_h(n,j',k') \nonumber\\
 + \sum_{k'=1}^{N_p} B(k',k)^T \mu_h(n,j,k') &= \sum_{k'=1}^{N_v}   L(k,k') \vv_h(n,j,k') &\nonumber\\
\sum_{k'=1}^{N_v} B(k,k') \pp_h(n,j,k') &= 0  && \\
\mbox{for} \quad k \in \Upsilon_{wall}\cup \Upsilon_{in}: \qquad \pp_h(n,j,k)&=0  \nonumber\\
\pp_h(N_t,j,k) &= 0, \nonumber\\
\nonumber&\\
\label{kkt_discr3}
\beta \sum_{k'=1}^{N_v} M(k,k') \uu_h(n,j,k')+ \sum_{k'=1}^{N_v} M_{cnt}(k,k') \pp_h(n,j,k') &=0,  &&
\end{alignat}
where $M_{cnt} = F_d + F_b$.
This system can be written in a compact matrix form
\begin{align}
\label{stokespa}
 \underbrace{\left[\begin{array}{ccc}
\bm{M}_1 & 0 & -\bm{K}^{*} \\
0 & \beta\bm{M}_2 & \bm{M_3}^T \\
-\bm{K} & \bm{M_3} & 0 \\
\end{array}\right]}_{:=\mathfrak{A}} 
\underbrace{\left[\begin{array}{c}
{\bf v}_h \\
{ p}_h \\
{\bf u}_h \\
{{\boldsymbol{\lambda}}}_h \\
{\mu}_h \\
\end{array}\right]}_{:=\bf y}=\underbrace{\left[\begin{array}{c}
{\bf 0} \\
{\bf 0} \\
{\bf g}_h \\
\end{array}\right]}_{:=\bf b}.
\end{align}
Since we use tensor product basis functions,
each of the block matrices in $\mathfrak{A}$ can be represented via Kronecker products.
In particular, the forward and adjoint operators for the time-stochastic-space Navier-Stokes equations write as
\begin{eqnarray}
\label{eq:K3}
\bm{K} & = & I_{N_t} \otimes I_{n_\xi}^{\otimes m} \otimes \mathcal{L}  + C\tau^{-1} \otimes I_{n_\xi}^{\otimes m} \otimes  \mathcal{M}  +  \bm{N}[\bar \vv_h], \\
\bm{K}^* & = & I_{N_t} \otimes I_{n_\xi}^{\otimes m} \otimes \mathcal{L}  + C^\top\tau^{-1} \otimes I_{n_\xi}^{\otimes m} \otimes  \mathcal{M}  +  \bm{W}[\bar \vv_h] - \bm{N}[\bar \vv_h],
\end{eqnarray}
where \cite{BSOS215}
\begin{itemize}
\item $I_{N_t}$ is the identity  matrix and $C=\mathrm{tridiag}(-1,1,0)$
comes from the implicit Euler discretization;
\item $I_{n_\xi}$ is the identity mass matrix for each variable $\xi$.
\item $\mathcal{L} = \begin{bmatrix}\nu L & B^\top \\ B & 0 \end{bmatrix}$ is the discretization of the stationary Stokes operator,
and $\mathcal{M}=\mathrm{blkdiag}(M, 0)$ for the velocity mass matrix, extended by zeros to the pressure space;
\item $\bm{N}$ and $\bm{W}$ correspond to the convection terms. They depend on the {\it low-rank} structure of the solution $\bar\vv_h$,
which will be introduced in the next section.
\end{itemize}
Furthermore, we introduce an extended velocity Laplace matrix $\mathcal{L}_0 = \mathrm{blkdiag}(L, 0)$.
Then we have
\begin{align}
\label{stoms}
 \bm{M}_1 &=   I_{N_t}\otimes I_{n_\xi}^{\otimes m}\otimes \mathcal{L}_0, &
 \bm{M}_2 &=   I_{N_t}\otimes I_{n_\xi}^{\otimes m} \otimes M, & \bm{M}_3 & = I_{N_t} \otimes I_{n_\xi}^{\otimes m}\otimes \begin{bmatrix}M_{cnt} \\ 0\end{bmatrix},
\end{align}
where
\begin{enumerate}
 \item $M_{cnt}=M$ for the distributed control, and
 \item $M_{cnt}(k,k') = \left\{\begin{array}{ll}M(k,k'), & \mbox{if }k,k' \in \Upsilon_{cnt}, \\ 0, & \mbox{otherwise}, \end{array}\right.$ for the boundary control.
\end{enumerate}

%

The third vector ${\bf g}_h$ in the right-hand side ${\bf b}$ in \eqref{stokespa} depends on the boundary conditions.
The KLE yields the following tensor form for the inflow function,
\begin{equation}
\boldsymbol\theta_h = \mathbf{g}_t \otimes \sum_{k=0}^{m} \left[ \left(\bigotimes_{\ell=1}^{k-1} \mathbf{e}_1\right) \otimes \mathbf{e}_2 \otimes \left(\bigotimes_{\ell=k+1}^{m} \mathbf{e}_1 \right) \otimes \boldsymbol\theta_{kh}\right],
\label{eq:inflow-discr}
\end{equation}
where $\mathbf{g}_t$ is the vector of the $1-\exp(-t)$ factor sampled at all time points, $\mathbf{e}_1$ and $\mathbf{e}_2$ are the first and second unit vectors, respectively, and $\boldsymbol\theta_{kh}$ is the discretization of the $k$-th spatial factor in \eqref{eq:kle_in}.
Now, distinguishing inner and boundary degrees of freedom, we partition the PDE matrix into the corresponding blocks,
$$
\bm{K} = \begin{bmatrix} \bm{K}_{II} & \bm{K}_{IB} \\ \bm{K}_{BI} & \bm{K}_{BB} \end{bmatrix}.
$$
The actual separation of inner (I) and boundary (B) elements is done in the spatial factors only.
Then the usual finite element approach is employed: we eliminate $\bm{K}_{IB}$ and $\bm{K}_{BI}$ in the left hand side of \eqref{stokespa},
replacing $\bm{K}$ by $\bm{K}_{II}$,
and construct the right-hand side from \eqref{eq:inflow-discr} as $\mathbf{g}_h = \bm{K}_{IB} \boldsymbol\theta_h.$

\section{Tensor Train decomposition approach}
\label{tt_decom} 

The solution $y(t,\xi_1,\ldots,\xi_m,{\bf x})$ (where $y$ stands for $\vv,p,{\bf u},\pp$ or $\mu$) is a multivariate function.
After discretization, independently in each variable $t,\xi_1,\ldots,\xi_m,{\bf x}$,
the discrete values of $y$ can be enumerated by $D=m+2$ independent 
indices\footnote{We don't separate different \emph{components} of ${\bf x}$,
which are actually \emph{dependent}, due to the domain geometry. 
Therefore, all spatial degrees of freedom are treated as one variable.}, i.e. they form a \emph{tensor}.
Storing such a tensor directly might be prohibitively expensive.
Therefore, we \emph{decompose} tensors using the separation of variables.

Given a $D$-index tensor $y(i_1,\ldots,i_D)$, its Tensor Train decomposition \cite{osel-tt-2011} is written as follows:
\begin{equation}
y(\mathbf{i}) = \sum\limits_{s_1=1}^{r_1} \cdots \sum\limits_{s_{D-1}=1}^{r_{D-1}} y^{(1)}_{s_1}(i_1) y^{(2)}_{s_1,s_2}(i_2) \cdots y^{(D)}_{s_{D-1}}(i_D),
\label{eq:tt}
\end{equation}
where $\mathbf{i}$ denotes the multi-index, $\mathbf{i} = (i_1,\ldots,i_D)$.
Each element of $y$ is represented (or approximated) by a sum of products of elements of smaller tensors $y^{(\ell)}$, $\ell=1,\ldots,D$, called TT blocks.
The auxiliary summation indices $s_1,\ldots,s_{D-1}$ are called rank indices, and their ranges $r_1,\ldots,r_{D-1}$ are called TT ranks.
The TT decomposition is also known as the Matrix Product States \cite{schollwock-2005,PerezGarcia-mps-2007}, since, omitting $s_1,\ldots,s_{D-1}$, we can say that an element $y(\mathbf{i})$ is equal to a product of $i_\ell$-dependent matrices.
The TT ranks depend on the particular tensor and accuracy, if \eqref{eq:tt} is satisfied approximately.
Denoting upper bounds $r_\ell \lesssim r$, $i_\ell\lesssim N$, $\ell=1,\ldots,D-1$, we estimate the storage cost of the TT blocks to be in $\mathcal{O}(DNr^2)$, which can be much less than the full amount $N^D$.
Usually, we refer to $r$ as the maximal TT rank.

The TT format allows to compress a given tensor up to a threshold $\varepsilon$, 
also if $y$ is already given in the TT format, but possibly with overestimated TT ranks.
This happens if we sum two arrays in the TT format, or perform the matrix multiplication.
A matrix $A$, acting on tensors of type $y$, can be seen as a $2D$-dimensional tensor and represented in a slightly different TT format,
\begin{equation}
A(\mathbf{i},\mathbf{j}) = \sum\limits_{s_1,\ldots,s_{D-1}=1}^{R_1,\ldots,R_{D-1}} A^{(1)}_{s_1}(i_1,j_1) \cdots  A^{(D)}_{s_{D-1}}(i_D,j_D),
\label{eq:ttm}
\end{equation}
where $\mathbf{i},\mathbf{j}$ are multi-indices consisting of $i_\ell$ and $j_\ell$, respectively.
This is consistent with the Kronecker product ($\otimes$) if $D=2$ and $R=1$.
The matrix-vector product $\sum_{\mathbf{j}} A(\mathbf{i},\mathbf{j}) y(\mathbf{j})$ can then be implemented in the TT format block by block \cite{osel-tt-2011,schollwock-2005}.

\subsection*{TT format of the convection matrix}
We assume now that the velocity, in particular the previous Picard iterate, is represented by a TT decomposition \eqref{eq:tt},
\begin{equation}
\bar\vv_h(n, j_1,\ldots,j_m, k) = \bar\vv^{(1)}(n) \bar\vv^{(2)}(j_1) \cdots \bar\vv^{(m+1)}(j_m) \bar\vv^{(m+2)}(k),
\label{eq:hat_v}
\end{equation}
with the ranks $r_1,\ldots,r_{m+1}$.
We are going to derive a matrix TT format \eqref{eq:ttm} for $\bm{N}[\bar\vv_h]$ from \eqref{eq:K3}.
First, we notice from \eqref{eq:N_detailed} that $\bm{N}$ depends linearly on $\bar\vv_h$.
Therefore, we can plug \eqref{eq:hat_v} into \eqref{eq:N_detailed} and distribute the summations.
This gives
\begin{eqnarray}
N_{0,n}(jk,j'k') &  = & \bar\vv^{(1)}(n) \\
            \label{eq:N_tt_xi}
               &\cdot & \left[\sum_{j''=0}^{N_\xi-1} \int_{\Gamma} \psi_j(\xi)\psi_{j'}(\xi)\psi_{j''}(\xi) \rho(\xi)d\xi \cdot  \bar\vv^{(2)}(j''_1) \cdots \bar\vv^{(m+1)}(j''_m) \right] \\
               &\cdot & \left[\sum_{k''=1}^{N_v} \int_{\mathcal{D}} \phi_k({\bf x}) \nabla \phi_{k'}({\bf x}) \phi_{k''}({\bf x}) dx \cdot \bar\vv^{(m+2)}(k'') \right].
\end{eqnarray}
The last term in this expression is a deterministic convection matrix $N_s\left[\bar\vv^{(m+2)}\right]$ \eqref{eq:N_detailed},
which can be assembled with the cost $\mathcal{O}(r_{m+1} N_v)$.
The middle term can be further factorised due to the tensor product PCE space $\mathcal{Y}_{n_\xi}$.
Indeed, introducing the matrices of triple products for all $k=1,\ldots,m,$
\begin{equation}
H_{j_k''}(j_k,j_k') = \int\limits_{\Gamma_k} \psi_{j_k}(\xi_k)\psi_{j_k'}(\xi_k)\psi_{j_k''}(\xi_k) \rho_k(\xi_k)d\xi_k, \quad H_{j_k''} \in \mathbb{R}^{n_\xi \times n_\xi},
\label{eq:H3}
\end{equation}
we see that the multi-dimensional triple product in \eqref{eq:N_tt_xi} can be written as a Kronecker product of $m$ individual matrices \eqref{eq:H3}.
Finally, we notice that $N_{0,n}$ acts independently on each time step $n$, i.e. it is diagonal w.r.t. $n$, and so is $\bm{N}$.
Overall, given the TT format of $\bar\vv$ \eqref{eq:hat_v}, there exists a TT format for the convection matrix with the same TT ranks,
\begin{equation}
\begin{split}
\bm{N}[\bar\vv_h] &= \sum_{s_1,\ldots,s_{m+1} =1}^{r_1,\ldots,r_{m+1}}\mathrm{diag}(\bar\vv^{(1)}_{s_1}) \\
       & \otimes  \left[\sum_{j_1=0}^{n_\xi-1}H_{j_1}\bar\vv^{(2)}_{s_1,s_2}(j_1)\right] \otimes \cdots \otimes \left[\sum_{j_{m}=0}^{n_\xi-1}H_{j_{m}}\bar\vv^{(m+1)}_{s_{m},s_{m+1}}(j_{m})\right] \\
       & \otimes  \mathrm{blkdiag}\left(N_s\left[\bar\vv^{(m+2)}_{s_{m+1}}\right], N_s\left[\bar\vv^{(m+2)}_{s_{m+1}}\right]\right).
\end{split}
\label{eq:conv_tt}
\end{equation}
Notice that the right-hand side $\mathbf{g}_h = \bm{K}_{IB} \boldsymbol\theta_h$ in \eqref{stokespa} must be recomputed in every Picard iteration, since $\bm{K}_{IB}$ carries the corresponding (new) part $\bm{N}_{IB}$.
The adjoint convection matrix can be constructed similarly, with $W_s$ defined in \eqref{eq:Ws},
\begin{equation*}
\begin{split}
\bm{W}[\bar\vv_h] &= \sum_{s_1,\ldots,s_{m+1} =1}^{r_1,\ldots,r_{m+1}}\mathrm{diag}(\bar\vv^{(1)}_{s_1}) \\
       & \otimes  \left[\sum_{j_1=0}^{n_\xi-1}H_{j_1}\bar\vv^{(2)}_{s_1,s_2}(j_1)\right] \otimes \cdots \otimes \left[\sum_{j_{m}=0}^{n_\xi-1}H_{j_{m}}\bar\vv^{(m+1)}_{s_{m},s_{m+1}}(j_{m})\right]
        \otimes  W_s\left[\bar\vv^{(m+2)}_{s_{m+1}}\right].
\end{split}
\end{equation*}

\subsection*{Alternating linear solver}
The Oseen equation  given by \eqref{stokespa} is a large linear system for ${\bf y}$, and needs to be solved keeping all the components in the TT format in order to keep the storage requirements low.
A state of the art approach to this problem is alternating tensor product algorithms \cite{white-dmrg-1993,schollwock-2005,holtz-ALS-DMRG-2012}.
Given the system $\mathfrak{A}{\bf y}={\bf b}$, we
iterate over $\ell=1,\ldots,D$, and seek only the elements of $y^{(\ell)}$ in each step, while the other TT blocks are fixed.
Notice that the TT format \eqref{eq:tt} is \emph{linear} with respect to the elements of each $y^{(\ell)}$, i.e., there exists a matrix $Y_\ell$ \eqref{eq:frame} such that ${\bf y} = Y_\ell y^{(\ell)}$.
This renders $\mathfrak{A}{\bf y}={\bf b}$ an overdetermined system $\mathfrak{A} Y_\ell y^{(\ell)}={\bf b}$ w.r.t. the elements of $y^{(\ell)}$.
This system is resolved via a projection onto $Y_\ell$, such that $y^{(\ell)}$ is computed from a smaller system $\left(Y_\ell^\top \mathfrak{A} Y_\ell\right) y^{(\ell)} = Y_\ell^\top {\bf b}$.
Particularly efficient realizations are the Density Matrix Renormalization Group (DMRG) methods from quantum physics \cite{white-dmrg-1993,jeckelmann-dmrgsolve-2002} and the Alternating Minimal Energy (AMEn) algorithm \cite{ds-amen-2014} from the mathematical community.
The DMRG approach computes two neighboring TT blocks in each step, say, $\ell$ and $\ell+1$, which allows to adapt the TT rank $r_\ell$ to the desired accuracy.
While the  DMRG method was found to be extremely effective for spin Schroedinger eigenvalue problems, it may deliver insufficient accuracy for linear systems, especially with non-symmetric matrices.
The AMEn method is usually faster and more robust, since it seeks only one TT block in each step, but then performs an explicit augmentation of the computed TT block of the solution by a TT block of the current residual.
This allows to change TT ranks and facilitate convergence.

However, the standard AMEn method is not applicable to $\mathfrak{A}$ directly: due to indefiniteness of $\mathfrak{A}$, its Galerkin projection may be degenerate.
For example, consider
$$
\mathfrak{A} = \begin{bmatrix}1 & 0 & 1 \\ 0 & 1 & 1 \\ 1 & 1 & 0\end{bmatrix} \quad \mbox{and} \quad Y_\ell = \begin{bmatrix}0 \\ 0 \\ 1\end{bmatrix}.
$$
One can readily verify that $Y_\ell^\top \mathfrak{A} Y_\ell=0$.
To cope with this issue, we employ the so-called block TT format \cite{dkos-eigb-2014} and project each \emph{submatrix} of $\mathfrak{A}$ separately.

Let us enumerate the components of \eqref{stokespa} as $y_{\iota}$, where
$y_1 = \vv_h$, $y_2=p_h$, $y_3={\bf u}_h$, $y_4=\pp_h$ and $y_5=\mu_h$.
We approximate all components simultaneously by a TT format with the same blocks except the $\ell$-th one for some $\ell=1,\ldots,D$, where the enumerator $\iota=1,\ldots,5$ appears,
\begin{equation}
y_{\iota}(\mathbf{i}) = \sum_{s_1,\ldots,s_{D-1}} y^{(1)}_{s_1}(i_1) \cdots y^{(\ell)}_{s_{\ell-1,s_\ell}}(i_\ell,\iota) \cdots y^{(D)}_{s_{D-1}}(i_D).
\label{eq:btt}
\end{equation}
Using an SVD of the core $y^{(\ell)}$, we can turn \eqref{eq:btt} into the form
$$
y_{\iota}(\mathbf{i}) = \sum_{s_1,\ldots,s_{D-1}} y^{(1)}_{s_1}(i_1) \cdots y^{(\ell)}_{s_{\ell-1},s_\ell}(i_\ell) \cdot y^{(\ell+1)}_{s_{\ell},s_{\ell+1}}(i_{\ell+1},\iota) \cdots y^{(D)}_{s_{D-1}}(i_D)
$$
or vice versa.
That is, we can \emph{move} $\iota$ to any particular TT block in the course of the alternating iteration \cite{dkos-eigb-2014}.
When the $\ell$-th block carries $\iota=1,\ldots,5$, the Oseen system \eqref{stokespa} is projected onto the other TT blocks component by component, giving a system
\begin{equation}
\label{stokespa-red}
\begin{bmatrix}
\hat M_1 & 0 & -\hat K^{*} \\
0 & \beta \hat M_2 & \hat M_3^T \\
-\hat K & \hat M_3 & 0 \\
\end{bmatrix}
y^{(\ell)}=
\begin{bmatrix}
0 \\
0 \\
\hat g \\
\end{bmatrix}
\end{equation}
on the components of the current TT block.
Here, $\hat A = Y_\ell^T {\bm A} Y_\ell$ for $\bm A \in \{\bm{K},\bm{M}_1,\bm{M}_2,\bm{M}_3\}$ are the submatrices projected onto the Galerkin basis $Y_\ell$ composed from the frozen TT blocks,
\begin{equation}
\begin{split}
Y_\ell(\mathbf{i},s_{\ell-1} i'_\ell s_\ell) & = y^{(1)}(i_1) \cdots y^{(\ell-1)}_{:,s_{\ell-1}}(i_{\ell-1}) \cdot I_{n_\ell}(i_\ell,i_\ell') \cdot y^{(\ell+1)}_{s_\ell,:}(i_{\ell+1}) \cdots y^{(D)}(i_D),
\end{split}
\label{eq:frame}
\end{equation}
where ``$:$'' stands for the full index range.

Under certain assumptions on the original system, one can prove that the block-reduced system \eqref{stokespa-red} is not degenerate.
\begin{theorem}
Suppose that the symmetric parts of $\bm{K}$ and $\bm{K}^*$ are positive, $\bm{K}+\bm{K}^\top>0$, $\bm{K}^*+\bm{K}^{*\top}>0$.
Then the reduced matrix in \eqref{stokespa-red} is invertible.
\end{theorem}
\begin{proof}
Due to the Poincar\'e theorem, the eigenvalues of an orthogonal projection of a symmetric matrix interlace with the eigenvalues of the original matrix.
In particular,
$$
\lambda_{\min}(\hat K + \hat K^\top) = \lambda_{\min}\left(Y_k^\top (\bm{K}+\bm{K}^\top) Y_k \right) \ge \lambda_{\min}(\bm{K}+\bm{K}^\top)>0.
$$
We also use that $Y_k$ is orthogonal in the AMEn method.
So, the symmetric part of $\hat K$ (as well as of $\hat K^*$) is positive.
Moreover, by the same interlace theorem we have that $\hat M_1 \ge 0$ and $\hat M_2>0$.
Now employ \cite[Theorem 3.2]{BenGolLie05},
which says that sufficient conditions for the KKT matrix to be invertible are that
the matrix of constraints $\begin{bmatrix}-\hat K & \hat M_3\end{bmatrix}$ is full rank, and
$$
\mathrm{ker}\begin{bmatrix}\hat M_1 & 0 \\ 0 & \beta\hat M_2\end{bmatrix} \cap \mathrm{ker}\begin{bmatrix}-\hat K & \hat M_3\end{bmatrix} = \{0\}.
$$
The first condition is fulfilled since $\hat K$ is invertible.
To verify the second criterion, consider a vector in the kernel of the constraints, which has the form
$$
w = \begin{bmatrix}y \\ u\end{bmatrix} = \begin{bmatrix}\hat K^{-1} \hat M_3 u \\ u\end{bmatrix}, \qquad \forall u \neq 0.
$$
Now check if it belongs to the kernel of the other matrix:
$$
w^\top \begin{bmatrix}\hat M_1 & 0 \\ 0 & \beta\hat M_2\end{bmatrix} w =
u^\top \hat M_3^\top \hat K^{-\top} \hat M_1 \hat K^{-1} \hat M_3 u + \beta u^\top \hat M_2 u > 0,
$$
since $u^\top \hat M_2 u>0$ while the first term is non-negative.
Together with positive semi-definiteness of $\mathrm{blkdiag}(\hat M_1,~ \beta\hat M_2),$ this yields that $w$ is not in its kernel.

\hfill\end{proof}

The second modification to the AMEn algorithm is the unification of sizes and elimination of pressures.
Note that the last ($D$-th) TT block carries the spatial degrees of freedom,
and different components of $y_{\iota}$ have different numbers of spatial basis functions: $N_v$ for $\vv_h$, ${\bf u}_h$ and $\pp_h$, and $N_p$ for $p_h$ and $\mu_h$.
The Galerkin matrix $Y_\ell$ has always the same number of rows; we make it $N_t n_\xi^m N_v$, i.e. consistent with $\vv_h$, ${\bf u}_h$ and $\pp_h$.

The state and adjoint pressures are on the one hand more difficult, since they are of different sizes, and moreover make $\bm{K}$ itself saddle-point, but on the other hand they are only connected with the other components via a Kronecker-rank-1 matrix $\bm{B} = I_{N_t} \otimes I_{n_\xi}^{\otimes m} \otimes B$.
Therefore, we can compute the pressure components only within the $D$-th, the spatial, block.
When we proceed to $\ell<D$ (time and stochastic variables), we eliminate the pressures in a Gauss-Seidel fashion: having originally
$
\bm{K} = \begin{bmatrix}\bm{A} & \bm{B}^T \\ \bm{B} & 0\end{bmatrix},
$
we project only the velocity part $\hat K = Y_\ell^T \bm{A} Y_\ell$ in the left hand side,
and cast $Y_\ell^T \bm{B}^T p_h$ and $Y_\ell^T \bm{B}^T \mu_h$ to the right hand side.
The consequence is two-fold: the velocity part $\bm{A}$ is positive definite, hence so is $\hat K$, and all sizes are now consistent.
%
%
%
%

\subsection*{Preconditioning}
Even the reduced system \eqref{stokespa-red} can still be rather large: for $\ell=D$ for example, it is of size $(3 N_v + 2N_p) r_{m+1}$.
Therefore, we use GMRES with the matching preconditioner \cite{SW13,BSOS215}, based on approximating the Schur complement.
The reduced matrix admits a (straightforwardly verifiable) decomposition
\be
\label{shurfac}
\begin{bmatrix}
\hat M_1 & 0 & -\hat K^{*} \\
0 & \beta \hat M_2 & \hat M_3^T \\
-\hat K & \hat M_3 & 0 \\
\end{bmatrix} =
\begin{bmatrix}
I & * & * \\
 & I  \\
 &  & I \\
\end{bmatrix}
\begin{bmatrix}
 & & -\hat S  \\
 & \beta \hat M_2 & \hat M_3^T  \\
-\hat K & \hat M_3 & 0 \\
\end{bmatrix},
\e
where $\hat S = \hat K^* + \hat M_1 \hat K^{-1} \hat M_3 (\beta \hat M_2)^{-1} \hat M_3^T$.
Next, we use the second (anti-triangular) factor in  (\ref{shurfac}) as a preconditioner.
However, $\hat K$ and $\hat S$ must be approximated to make solving linear systems with them feasible.

The first matrix has the Kronecker form
$$
\hat K = \hat I \otimes \mathcal{L} + \hat C \otimes \mathcal{M} + \sum_{s=1}^{r_{m+1}} \hat D_s \otimes \mathcal{N}_s,
$$
where $\hat I$, $\hat C$ and $\hat D_s$ are $I_{N_t} \otimes I_{n_\xi}^{\otimes m}$, $C\tau^{-1} \otimes I_{n_\xi}^{\otimes m}$ and the parts of the convection matrix \eqref{eq:conv_tt}, respectively, projected via $Y_{m+2}$.
We approximate $\hat K$ by the following Sylvester operator:
$$
\tilde K = \hat I \otimes \left(\mathcal{L} + \sum_{s=1}^{r_{m+1}} \mathbb{E}\lambda(\hat D_s) \cdot \mathcal{N}_s \right) + \hat C \otimes \mathcal{M},
$$
where $\mathbb{E}\lambda(\hat D_s)$ is an average of the eigenvalues of $\hat D_s$ (it is symmetric, since so is the time-stochastic part of $\bm{N}$).
Now $\tilde K$ can be inverted by the Bartels-Stewart method, since $\hat C$ is small and can be easily Schur-factorized.

For the Schur complement $\hat S$ we use the matching factorisation,
\begin{equation}
\tilde S = \left(\hat K^* + \frac{1}{c} \hat M_1\right) \hat K^{-1} \left(\hat K + c \cdot \hat M_3 (\beta \hat M_2)^{-1} \hat M_3^T\right),
\label{eq:sc2}
\end{equation}
where $c>0$ is a normalization constant.
Traditionally it is proposed \cite{SW13,BSOS215} to take $c=\sqrt{\beta}$, and this is indeed an optimal choice for the distributed control.
For the boundary control, however, $\hat M_3 (\beta \hat M_2)^{-1} \hat M_3^T$ is significantly rank-deficient, and an optimal $c$ may differ.
We look for $c$ that minimizes the residual of the system $(\tilde S^{-1} \hat S) u = \tilde S^{-1} f$ with a random right-hand side $f$ after $5$ GMRES iterations.
Since it is difficult to differentiate the GMRES residual w.r.t. $c$, we employ the zero-order \emph{golden section} optimization algorithm \cite{Cheney}, initialized with an interval $\log_{10} c \in [\log_{10}\sqrt{\beta}-6, \log_{10}\sqrt{\beta}+6]$.
Fortunately, it is sufficient to perform this procedure in the first Picard iteration only, as the optimal $c$ does not seem to change in the latter.
To solve systems with $\tilde S$ efficiently, we also approximate the factors in brackets in \eqref{eq:sc2} by Sylvester matrices, similarly to $\tilde K$.

The full Algorithm \ref{alg} combines the standard Picard iteration for the Navier-Stokes equation \cite{ESW14} and the AMEn iteration in the block TT format \cite{BSOS215}.

\begin{algorithm}[t]
\caption{Block AMEn-Picard iteration for solving the stochastic inverse Navier-Stokes equations in the TT format}
\label{alg}
\begin{algorithmic}[1]
\State Initialize $\vv_h={\bf u}_h=\pp_h=0$, $p_h =\mu_h=0$.
\State Initialize the block TT format \eqref{eq:btt} with $\iota$ placed in $y^{(m+2)}$.
\For{iter=1,2,\ldots}
  \State Copy $\bar\vv_h = \vv_h, \bar p_h = p_h, \bar{\bf u}_h = {\bf u}_h, \bar\pp_h = \pp_h, \bar\mu_h = \mu_h$.
  \State Construct the convection matrices $\bm{N}[\bar\vv_h]$ and $\bm{W}[\bar\vv_h]$ \eqref{eq:conv_tt}.
  \State Construct and solve the projected system \eqref{stokespa-red} for $\ell=m+2$.
  \State Extract individual components $\vv_h=y_1,p_h=y_2,{\bf u}_h=y_3,\pp_h=y_4,\mu_h=y_5$.
  \If{$\|\vv_h-\bar\vv_h\|\le \eps \|\vv_h\|$, $\|p_h-\bar p_h\|\le \eps\|p_h\|$ and $\|{\bf u}_h-\bar {\bf u}_h\|\le \eps\|{\bf u}_h\|$}
    \State Stop.
  \EndIf
  \State Assemble $y_{\iota}$ from $y_1=\vv_h,y_2={\bf u}_h,y_3=\pp_h$ only.
  \State Apply the SVD to $y^{(m+2)}$ and move $\iota$ to $y^{(m+1)}$.
  \For{$\ell=m+1,m,\ldots,1,2,\ldots,m+1$}
    \State Construct and solve an analog of \eqref{stokespa-red} with fixed pressures,
      \begin{equation*}
      \begin{bmatrix}
      \hat M_1 & 0 & -\hat A^{*} \\
      0 & \beta \hat M_2 & \hat M_3^T \\
      -\hat A & \hat M_3 & 0 \\
      \end{bmatrix}
      y^{(\ell)}=
      \begin{bmatrix}
      -Y_\ell^T \bm{B}^T \mu_h \\
      0 \\
      \hat g - Y_\ell^T \bm{B}^T p_h \\
      \end{bmatrix}
      \end{equation*}
     \State Apply the SVD to $y^{(\ell)}$.
     \If{$\ell$ is increasing}
       \State Move $\iota$ to $y^{(\ell+1)}$.
     \Else
       \State Move $\iota$ to $y^{(\ell-1)}$.
     \EndIf
  \EndFor
\EndFor
\end{algorithmic}
\end{algorithm}

\section{Numerical results}
\label{experiments}
We implemented the computational codes on the basis of the Matlab TT-Toolbox \cite{tt-toolbox} and
the IFISS 3.3 toolbox \cite{ifiss3.4} and run on one core of the \texttt{otto} cluster at MPI Magdeburg,
an Intel Xeon X5650 @ 2.67GHz.
The main focus of this paper is the numerical scheme, hence we vary the model and discretization parameters one by one.
The default parameters, unless otherwise stated, are presented in Table \ref{tab:def}.

\begin{table}[h!]
\centering
\caption{Default model and discretization parameters}
\label{tab:def}
\begin{tabular}{ccccccccc}
$\eps$    & $\nu$ & $\beta$  &  $N_t$    & $T$  & $h$      & $n_\xi$ & $\gamma$ & $m$  \\
\hline
$10^{-4}$ & $0.1$ & $10^{-2}$&  $2^{10}$ & $30$ & $2^{-3}$ & $8$       & $4$      & $4$ \\
\end{tabular}
\end{table}

The natural indicators of computational complexity are the CPU time and the number of iterations, while the storage complexity of the TT representation is governed by the TT ranks.
We show the maximal TT rank and the memory saving ratio of the TT format, compared to the full tensor representation, i.e.
$$
\mbox{MemR} = \frac{N_t r_1 + \sum_{\ell=1}^m r_{\ell}r_{\ell+1} n_\xi + r_{m+1} N_v}{N_t n_\xi^m N_v}.
$$
The smaller the TT ranks, the smaller is this ratio, which indicates the storage savings.

As a quantity of interest, we compute the average squared norm of vorticity 
$\|\nabla \times \vv\|^2$, where the norm is taken in $L^2([0,T]) \times L^2(\Omega) \times L^2(\mathcal{D})$, i.e. the doubled first term of the functional \eqref{J}.

The largest problem was solved with $m=8,\; N_t=1024$ and $h=2^{-4};$ this value of $h$ gives $N_v=23042$ and $N_p=2945.$
Thus, the number of unknowns in the full KKT problem would be $(3N_v+2N_p)\cdot {n_\xi}^m\cdot N_t = 1.29\times 10^{15}.$
Note that the number of TT elements is $(3N_v+2N_p)r +  m{n_\xi}r^2 + N_tr;$
but since the first term is dominating,  it suffices to estimate the whole cost with the first term.
The maximum ranks are $70$ for the distributed control  and $120$  for the boundary control cases (see Figures \ref{fig:m} and \ref{fig:gamma_m_bnd} right, respectively).
This gives roughly $5$ and $9$ millions of unknowns.

\subsection{Distributed control}
\subsubsection{Discretization parameters}
In the first series of tests we consider the distributed control, applied in the right hand side of \eqref{kkt_cond1}.
We start with verifying correctness of the discretization schemes.
We show the results for different spatial mesh sizes in Fig. \ref{fig:nc}, and for different time step sizes in Fig. \ref{fig:nt}.

The number of spatial degrees of freedom grows quadratically with the number of steps in each dimension.
Since we use a direct solver for the spatial Stokes-like problems in \eqref{eq:sc2}, the computational complexity is higher;
in particular, it grows proportionally to $h^{-2.5}$ in Fig. \ref{fig:nc}.
To estimate the discretization error, we take the vorticity from the finest grid $h=2^{-4}$ as the reference value $\|\nabla \times \vv_*\|$,
and compare it with the values on coarser grids.
We observe the convergence order $h^{1.5}$,
which is reasonable for the backstep domain, 
containing a corner with an angle 
$3\pi/2;$  see e.g., \cite{babuska-regularity-1988,babuska-hp-1988}.

In order to keep the same number of stochastic variables 
throughout the spatial mesh test, we restrict it to $m=2$.
This is necessary since the higher KLE components 
in \eqref{eq:kle_in} are not resolved on the coarsest grid in 
this test,
and the problem would be in a severely pre-asymptotic regime 
otherwise.

The TT ranks grow very moderately with the grid sizes, with the asymptotic dependence close to logarithmic.
Since the full tensor storage, proportional to both $N_v$ and $N_t$ grows much faster, this leads actually to a decrease of the memory ratio with both spatial (Fig. \ref{fig:nc}) and time (Fig. \ref{fig:nt}) grid refinement.

The CPU time versus the time grid size demonstrates two regions.
For smaller numbers of time steps, solution of the spatial problems is dominating in the total complexity, while the TT ranks grow logarithmically.
This leads to a logarithmic growth of the CPU time as well.
When the number of time steps is large, the temporal problems become the most time consuming, and the CPU time starts
to grow linearly with $N_t$ (the actual estimated order is $1.2$ due to the additional logarithmic growth of the ranks).

The implicit Euler time discretization manifests an expected convergence of the first order.
Here the reference vorticity $\|\nabla \times \vv_\star\|$ is computed at the $N_t=2^{15}$ grid.

%



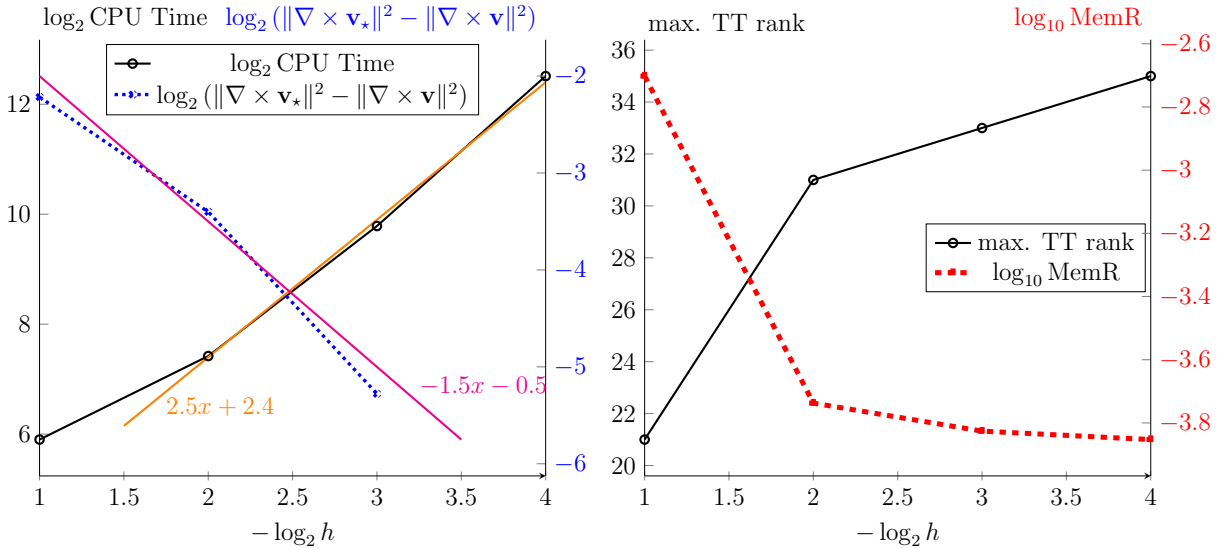
\begin{figure}[h!]
\caption{CPU time and convergence w.r.t. grid refinement (left) and TT rank and memory reduction ratio (right) for different spatial mesh sizes. Change of defaults: $m=2$.}
\label{fig:nc}
\resizebox{0.98\linewidth}{!}{
\begin{tikzpicture}
  \begin{axis}[%
  xmode=normal,
  ymode=normal,
  xmin=1,xmax=4,
  xlabel={$-\log_2 h$},
  legend style={at={(0.50,0.99)},anchor=north},
  ylabel=$\log_2\mbox{CPU Time}$, axis y line*=left, y label style={at={(-0.01,1.0)},anchor=south west,rotate=-90,color=black},
  cycle list name=stokescontrol,
  y filter/.code={\pgfmathparse{log10(\pgfmathresult)/log10(2.0)}\pgfmathresult},
  ]

  \addplot+[] coordinates {
   (1, 59.7220   )
   (2, 171.1626  )
   (3, 881.3946  )
   (4, 5.8359e+03)}; \addlegendentry{$\log_2\mbox{CPU Time}$}
   \addplot+[] coordinates {(0,1)}; \addlegendentry{$\log_2\left(\|\nabla \times \vv_{\star}\|^2 - \|\nabla \times \vv\|^2\right)$};

   \addplot+[line width=1pt,color=orange,solid,no marks,domain=1.5:4] {2^(2.5*x+2.4)};
   \node at (axis cs:1.7,6.5) [anchor=west,color=orange] {$2.5x+2.4$};
  \end{axis}
  \begin{axis}[%
  xmode=normal,
  ymode=normal,
  xmin=1,xmax=4,
  legend style={at={(0.01,0.01)},anchor=south west},
  ylabel=$\log_2\left(\|\nabla \times \vv_{\star}\|^2 - \|\nabla \times \vv\|^2\right)$, axis y line*=right, y label style={at={(1.0,1.0)},anchor=south east,rotate=-90,color=blue},every y tick label/.style={blue},
  axis x line=none,
  cycle list name=stokescontrol,
  y filter/.code={\pgfmathparse{log10(12.097289795 - \pgfmathresult)/log10(2.0)}\pgfmathresult},
  ]
  \pgfplotsset{cycle list shift=1};
  \addplot+[] coordinates {
   (1, 11.882014339)
   (2, 12.002581938)
   (3, 12.071493602)};

   \addplot+[line width=1pt,color=magenta,solid,no marks,domain=1:3.5] {12.097289795 - 2^(-1.5*x-0.5)};
   \node at (axis cs:3.2,-5.2) [anchor=west,color=magenta] {$-1.5x-0.5$};
  \end{axis}
\end{tikzpicture}
\begin{tikzpicture}
  \begin{axis}[%
  xmode=normal,
  ymode=normal,
  xmin=1,xmax=4,
  legend style={at={(0.99,0.50)},anchor=east},
  ylabel=max. TT rank, axis y line*=left, y label style={at={(-0.01,1.0)},anchor=south west,rotate=-90},
  xlabel={$-\log_2 h$},
  cycle list name=stokescontrol
  ]
  \addplot+[] coordinates {
   (1, 21)
   (2, 31)
   (3, 33)
   (4, 35)}; \addlegendentry{max. TT rank};
  \pgfplotsset{cycle list shift=1};
  \addplot+[] coordinates{(0,0)}; \addlegendentry{$\log_{10}\mbox{MemR}$}; 
  \end{axis}
  \begin{axis}[%
  xmode=normal,
  ymode=normal,
  legend style={at={(0.01,0.01)},anchor=south west},
  ylabel=$\log_{10}\mbox{MemR}$, axis y line*=right, y label style={at={(1.0,1.0)},anchor=south east,rotate=-90,color=red},every y tick label/.style={red},
  axis x line=none,
  cycle list name=stokescontrol,
  y filter/.code={\pgfmathparse{log10(\pgfmathresult)}\pgfmathresult},
  ]
  \pgfplotsset{cycle list shift=2};
  \addplot+[] coordinates {
   (1, 1.9847e-03)
   (2, 1.8313e-04)
   (3, 1.4929e-04)
   (4, 1.4039e-04)};
  \end{axis}
\end{tikzpicture}
}
\end{figure}

~

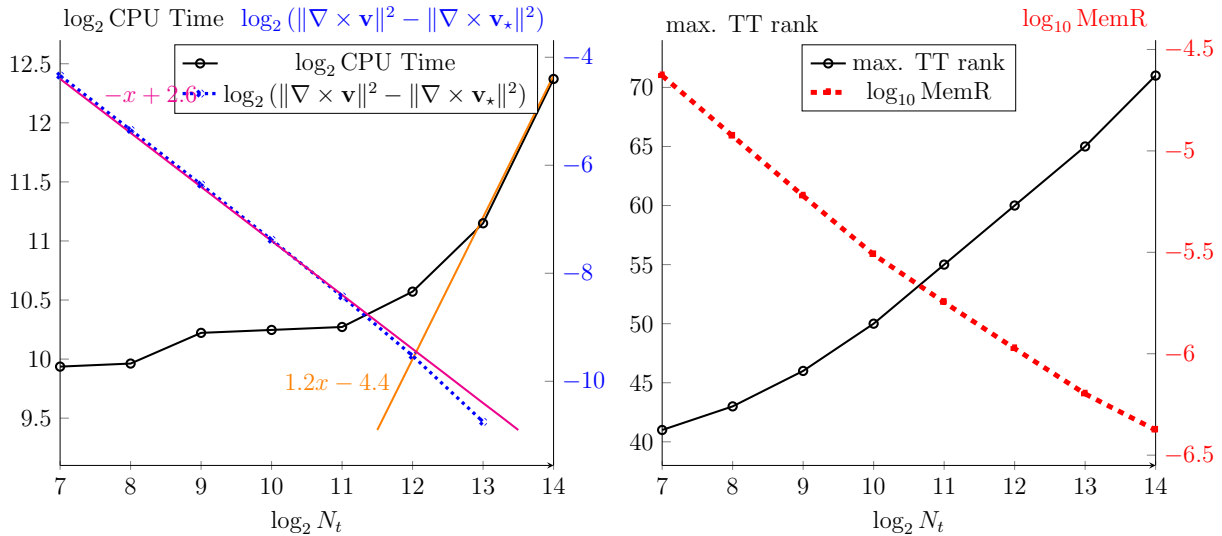
\begin{figure}[h!]
\caption{CPU time and convergence w.r.t. grid refinement (left) and TT rank and memory reduction ratio (right) for different numbers of time steps.}
\label{fig:nt}
\resizebox{0.98\linewidth}{!}{
\begin{tikzpicture}
  \begin{axis}[%
  xmode=normal,
  ymode=normal,
  xmin=7,xmax=14,
  xlabel={$\log_2 N_t$},
  legend style={at={(0.60,0.99)},anchor=north},
  ylabel=$\log_2\mbox{CPU Time}$, axis y line*=left, y label style={at={(-0.01,1.0)},anchor=south west,rotate=-90,color=black},
  cycle list name=stokescontrol,
  y filter/.code={\pgfmathparse{log10(\pgfmathresult)/log10(2)}\pgfmathresult},
  ]

  \addplot+[] coordinates {
   (7 , 978.8806  )
   (8 , 996.5336  )
   (9 , 1.1932e+03)
   (10, 1.2139e+03)
   (11, 1.2351e+03)
   (12, 1.5197e+03)
   (13, 2.2695e+03)
   (14, 5.2919e+03)}; \addlegendentry{$\log_2\mbox{CPU Time}$};
   \addplot+[] coordinates{(0,1)}; \addlegendentry{$\log_2\left(\|\nabla \times\vv\|^2 - \|\nabla \times\vv_{\star}\|^2\right)$}; 

   \addplot+[line width=1pt,color=orange,solid,no marks,domain=11.5:14] {2^(1.2*x-4.4)};
   \node at (axis cs:11.8,9.8) [anchor=east,color=orange] {$1.2x-4.4$};
  \end{axis}
  \begin{axis}[%
  xmode=normal,
  ymode=normal,
  xmin=7,xmax=14,
  legend style={at={(0.01,0.01)},anchor=south west},
  ylabel=$\log_2\left(\|\nabla \times\vv\|^2 - \|\nabla \times\vv_{\star}\|^2\right)$, axis y line*=right, y label style={at={(1.0,1.0)},anchor=south east,rotate=-90,color=blue},every y tick label/.style={blue},
  axis x line=none,
  cycle list name=stokescontrol,
  y filter/.code={\pgfmathparse{log10(\pgfmathresult - 12.066229586)/log10(2.0)}\pgfmathresult},
  ]
  \pgfplotsset{cycle list shift=1};
  \addplot+[] coordinates {
   (7 ,    12.115638555)
   (8 ,    12.090841716)
   (9 ,    12.078438998)
   (10,    12.072237403)
   (11,    12.069136609)
   (12,    12.067586210)
   (13,    12.066811001)};

   \addplot+[line width=1pt,color=magenta,solid,no marks,domain=7:13.5] {12.066229586 + 2^(-1*x+2.6)};
   \node at (axis cs:7.5,-4.7) [anchor=west,color=magenta] {$-x+2.6$};
  \end{axis}
\end{tikzpicture}
\begin{tikzpicture}
  \begin{axis}[%
  xmode=normal,
  ymode=normal,
  xmin=7,xmax=14,
  legend style={at={(0.50,0.99)},anchor=north},
  ylabel=max. TT rank, axis y line*=left, y label style={at={(-0.01,1.0)},anchor=south west,rotate=-90},
  xlabel={$\log_{2}N_t$},
  cycle list name=stokescontrol
  ]
  \addplot+[] coordinates {
   (7 , 41)
   (8 , 43)
   (9 , 46)
   (10, 50)
   (11, 55)
   (12, 60)
   (13, 65)
   (14, 71)}; \addlegendentry{max. TT rank};
   \pgfplotsset{cycle list shift=1};
   \addplot+[] coordinates{(0,0)}; \addlegendentry{$\log_{10}\mbox{MemR}$}; 
  \end{axis}
  \begin{axis}[%
  xmode=normal,
  ymode=normal,
  legend style={at={(0.01,0.01)},anchor=south west},
  ylabel=$\log_{10}\mbox{MemR}$, axis y line*=right, y label style={at={(1.0,1.0)},anchor=south east,rotate=-90,color=red},every y tick label/.style={red},
  axis x line=none,
  cycle list name=stokescontrol,
  y filter/.code={\pgfmathparse{log10(\pgfmathresult)}\pgfmathresult},
  ]
  \pgfplotsset{cycle list shift=2};
  \addplot+[] coordinates {
   (7 , 2.3510e-05)
   (8 , 1.1838e-05)
   (9 , 6.0018e-06)
   (10, 3.0838e-06)
   (11, 1.7945e-06)
   (12, 1.0611e-06)
   (13, 6.3420e-07)
   (14, 4.2072e-07)};
  \end{axis}
\end{tikzpicture}
}
\end{figure}

\subsubsection{Robustness with respect to model parameters: viscosity, KLE decay rate and dimension}
Having checked the discretization schemes, we can vary the other model parameters and see how the performance and the quantity of interest behave.
In Fig. \ref{fig:nu} we vary the viscosity from $10^{-1}$ to $10^{-4}$ and track the computational time and the number of Picard iterations (left), as well as the storage indicators and the quantity of interest (right).
We see that the CPU time grows only logarithmically with the Reynolds number, and $3$ Picard iterations are enough for the nonlinear system to converge.
This is also the case for other experiments with a \emph{sufficient} control, i.e. with the distributed control and small enough $\beta$, that was observed previously \cite{SSNS17} for a deterministic problem as well.
The TT rank (and hence memory ratio) and the vorticity remain at almost the same level (compared to variation of other parameters and other tests), which indicates that the minimal-vorticity solution is similar to the Stokes flow, fully defined by the domain and boundary conditions, rather than the viscosity.



\begin{figure}[h!]
\caption{CPU time and the number of Picard iterations (left) and TT 
rank, memory reduction ratio and the vorticity (right) 
for different kinematic viscosities.}
\label{fig:nu}
\resizebox{0.98\linewidth}{!}{
\begin{tikzpicture}
  \begin{axis}[%
  xmode=normal,
  ymode=normal,
  xmin=1,xmax=4,
  legend style={at={(0.99,0.01)},anchor=south east},
  xlabel={$-\log_{10}\nu$},
  ylabel=CPU Time, axis y line*=left, y label style={at={(-0.01,1.0)},anchor=south west,rotate=-90,color=black},
  cycle list name=stokescontrol
  ]
  \addplot+[] coordinates {
   (1,              1.1962e+03 )
   (log10(50.0),    1.4576e+03 )
   (2,              1.6360e+03 )
   (log10(500.0),   2.0184e+03 )
   (3,              2.1673e+03 )
   (log10(5000.0),  2.4554e+03 )
   (4,              2.5997e+03 )}; \addlegendentry{CPU Time};
   \addplot+[] coordinates {(0,0)}; \addlegendentry{Picard iterations};
  \end{axis}
  \begin{axis}[%
  xmode=normal,
  ymode=normal,
  ymin=2,ymax=4,
  ylabel=Picard iterations, axis y line*=right, y label style={at={(1.0,1.0)},anchor=south east,rotate=-90,color=blue},every y tick label/.style={blue},
  axis x line=none,
  cycle list name=stokescontrol
  ]
  \pgfplotsset{cycle list shift=1};
  \addplot+[] coordinates {
   (1,              3 )
   (log10(50.0),    3 )
   (2,              3 )
   (log10(500.0),   3 )
   (3,              3 )
   (log10(5000.0),  3 )
   (4,              3 )};
  \end{axis}
\end{tikzpicture}
\begin{tikzpicture}
  \begin{axis}[%
  xmode=normal,
  ymode=normal,
  ymin=40,ymax=50,
  legend style={at={(0.01,0.70)},anchor=west},
  ylabel=max. TT rank, axis y line*=left, y label style={at={(-0.01,1.0)},anchor=south west,rotate=-90},
  xlabel={$-\log_{10}\nu$},
  cycle list name=stokescontrol
  ]
  \addplot+[] coordinates {
   (1,              50 )
   (log10(50.0),    49 )
   (2,              49 )
   (log10(500.0),   49 )
   (3,              48 )
   (log10(5000.0),  48 )
   (4,              47 )}; \addlegendentry{max. TT rank};
   \pgfplotsset{cycle list shift=1};
   \addplot+[] coordinates {(0,0)}; \addlegendentry{$\log_{10}\mbox{MemR}$};
   \addplot+[] coordinates {(0,0)}; \addlegendentry{$\|\nabla \times \vv\|^2$};
  \end{axis}
  \begin{axis}[%
  xmode=normal,
  ymode=normal,
  ymin=-6,ymax=-5,
  legend style={at={(0.01,0.01)},anchor=south west},
  ylabel=$\log_{10}\mbox{MemR}$, axis y line*=right, y label style={at={(1.0,1.0)},anchor=south east,rotate=-90,color=red},every y tick label/.style={red},
  axis x line=none,
  cycle list name=stokescontrol,
  y filter/.code={\pgfmathparse{log10(\pgfmathresult)}\pgfmathresult},
  ]
  \pgfplotsset{cycle list shift=2};
  \addplot+[] coordinates {
   (1,              3.0838e-06 )
   (log10(50.0),    3.1408e-06 )
   (2,              3.1408e-06 )
   (log10(500.0),   3.1408e-06 )
   (3,              3.1408e-06 )
   (log10(5000.0),  3.1408e-06 )
   (4,              3.1408e-06 )};
  \end{axis}
  \begin{axis}[%
  xmode=normal,
  ymode=normal,
  ymin=12,ymax=12.3,
  legend style={at={(0.01,0.01)},anchor=south west},
  ylabel=$\|\nabla \times \vv\|^2$, axis y line*=right, y label style={at={(1.25,1.0)},anchor=south east,rotate=-90,color=green!50!black},
  axis x line=none,
  cycle list name=stokescontrol
  ]
  \pgfplotsset{cycle list shift=3};
  \pgfplotsset{every outer y axis line/.style={xshift=1cm}, every tick/.style={xshift=1cm}, every y tick label/.style={xshift=1cm,color=green!50!black} }
  \addplot+[] coordinates {
   (1,              12.072237403 )
   (log10(50.0),    12.072165240 )
   (2,              12.072162448 )
   (log10(500.0),   12.072161554 )
   (3,              12.072161527 )
   (log10(5000.0),  12.072161519 )
   (4,              12.072161519 )};
  \end{axis}
\end{tikzpicture}
}
\end{figure}
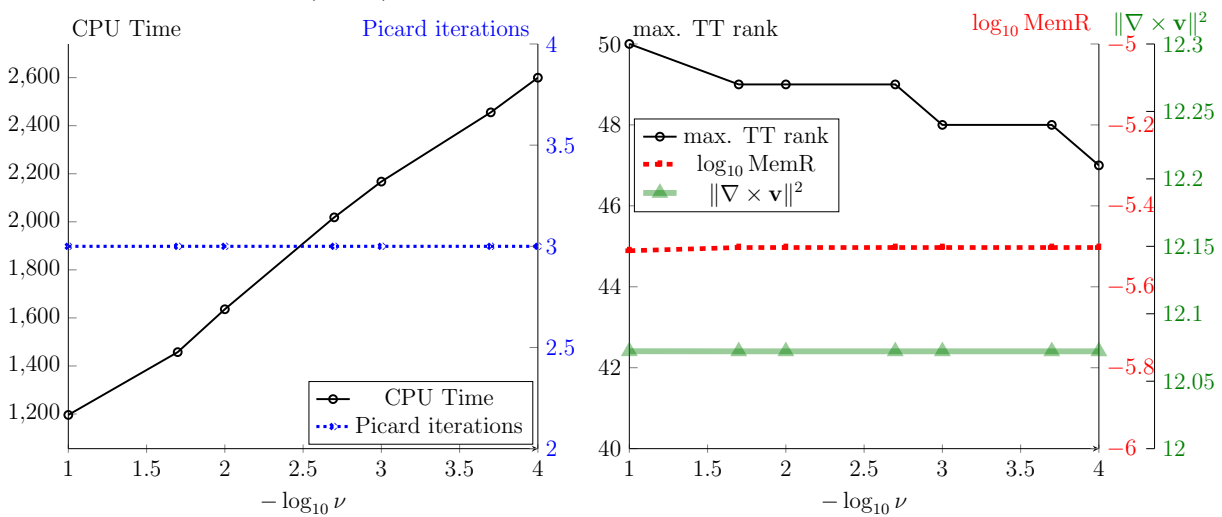

There are two parameters defining the ``randomness'' of the problem: the decay rate of the KLE coefficients $\gamma$, and the total number of KLE components after truncation $m$, see \eqref{eq:kle_in_new}.
The smaller is $\gamma$, the slower is the decay, and hence the larger is the number of ``effective'' dimensions.
This makes the problem more difficult to solve, as we observe in Fig. \ref{fig:gamma}.




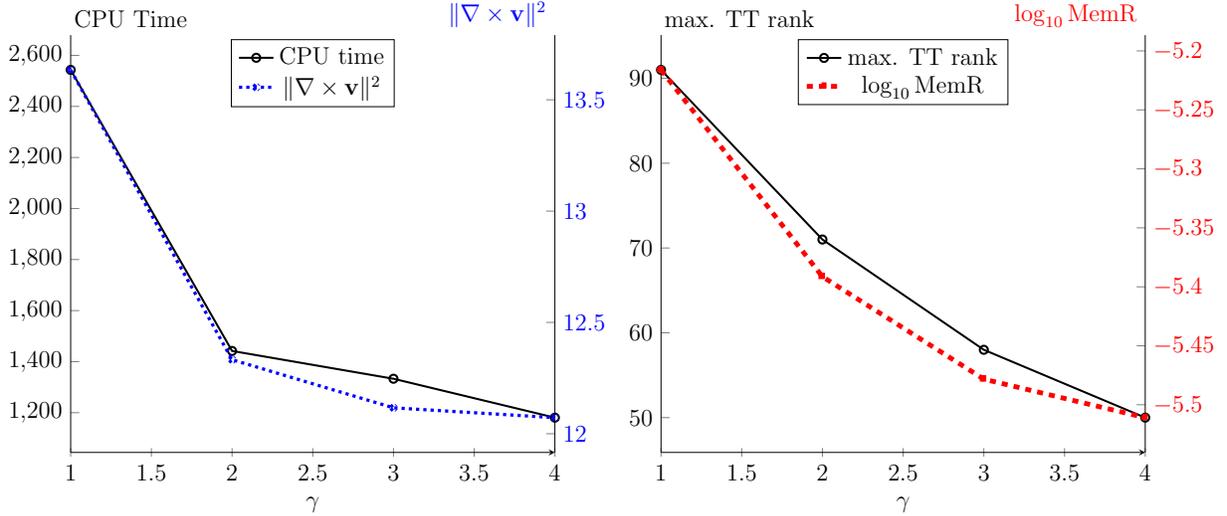
\begin{figure}[h!]
\caption{CPU time and the vorticity (left) and TT rank and
the memory reduction ratio (right) for different KLE decay rates $\gamma.$}
\label{fig:gamma}
\resizebox{0.98\linewidth}{!}{
\begin{tikzpicture}
  \begin{axis}[%
  xmode=normal,
  ymode=normal,
  xmin=1,xmax=4,
  xlabel={$\gamma$},
  legend style={at={(0.50,0.99)},anchor=north},
  ylabel=CPU Time, axis y line*=left, y label style={at={(-0.01,1.0)},anchor=south west,rotate=-90,color=black},
  cycle list name=stokescontrol,
  ]
  \addplot+[] coordinates {
   (1, 2.5431e+03)
   (2, 1.4419e+03)
   (3, 1.3330e+03)
   (4, 1.1806e+03)}; \addlegendentry{CPU time};
   \addplot+[] coordinates{(0,0)}; \addlegendentry{$\|\nabla \times\vv\|^2$}; 
  \end{axis}
  \begin{axis}[%
  xmode=normal,
  ymode=normal,
  legend style={at={(0.01,0.01)},anchor=south west},
  ylabel=$\|\nabla \times\vv\|^2$, axis y line*=right, y label style={at={(1.0,1.0)},anchor=south east,rotate=-90,color=blue},every y tick label/.style={blue},
  axis x line=none,
  cycle list name=stokescontrol,
  ]
  \pgfplotsset{cycle list shift=1};
  \addplot+[] coordinates {
   (1, 13.634238558)
   (2, 12.332533631)
   (3, 12.116461259)
   (4, 12.072237403)};
  \end{axis}
\end{tikzpicture}
\begin{tikzpicture}
  \begin{axis}[%
  xmode=normal,
  ymode=normal,
  xmin=1,xmax=4,
  legend style={at={(0.50,0.99)},anchor=north},
  ylabel=max. TT rank, axis y line*=left, y label style={at={(-0.01,1.0)},anchor=south west,rotate=-90},
  xlabel={$\gamma$},
  cycle list name=stokescontrol
  ]
  \addplot+[] coordinates {
   (1, 91)
   (2, 71)
   (3, 58)
   (4, 50)}; \addlegendentry{max. TT rank};
   \pgfplotsset{cycle list shift=1};
   \addplot+[] coordinates{(0,0)}; \addlegendentry{$\log_{10}\mbox{MemR}$}; 
  \end{axis}
  \begin{axis}[%
  xmode=normal,
  ymode=normal,
  legend style={at={(0.01,0.01)},anchor=south west},
  ylabel=$\log_{10}\mbox{MemR}$, axis y line*=right, y label style={at={(1.0,1.0)},anchor=south east,rotate=-90,color=red},every y tick label/.style={red},
  axis x line=none,
  cycle list name=stokescontrol,
  y filter/.code={\pgfmathparse{log10(\pgfmathresult)}\pgfmathresult},
  ]
  \pgfplotsset{cycle list shift=2};
  \addplot+[] coordinates {
   (1, 6.0802e-06)
   (2, 4.0608e-06)
   (3, 3.3245e-06)
   (4, 3.0838e-06)};
  \end{axis}
\end{tikzpicture}
}
\end{figure}


Increasing the number of stochastic variables faces the same issue as decreasing the number of spatial grid points: a coarse spatial grid may not resolve the latter KLE terms.
In order to vary $m$ up to $8$, we employ the spatial grid with $h=2^{-4}$.
The results are shown in Fig. \ref{fig:m}.
The CPU times grows milder than linearly with $m$ since (a) the most time-consuming stage is still the solution of the spatial problems, and (b) the model saturates with larger $m$ due to the decaying KLE series.
The latter phenomenon is also illustrated by convergence of the vorticity norm with $m$.

This example illustrates how the TT decomposition gets rid of the curse of dimensionality.
The total number of degrees of freedom in a full tensor would grow exponentially in $m$, while the TT ranks (and hence the TT storage) grow milder than linearly.
This leads to an exponential decay of the memory reduction ratio, which reaches 9 orders of magnitude for the largest dimension.

\begin{figure}[h!]
\caption{CPU time and the vorticity (left) and TT rank and the memory reduction ratio (right) for different numbers of stochastic variables $m$. Change of defaults: $h=2^{-4}$.}
\label{fig:m}
\resizebox{0.98\linewidth}{!}{
\begin{tikzpicture}
  \begin{axis}[%
  xmode=normal,
  ymode=normal,
  xmin=1,xmax=8,
  ytick={4,5,6,7,8,9,10},
  yticklabels={4000,5000,6000,7000,8000,9000,10000},
  xlabel={$m$},
  legend style={at={(0.40,0.01)},anchor=south},
  ylabel=CPU Time, axis y line*=left, y label style={at={(-0.01,1.0)},anchor=south west,rotate=-90,color=black},
  cycle list name=stokescontrol,
  y filter/.code={\pgfmathparse{0.001*\pgfmathresult}\pgfmathresult},
  ]
  \addplot+[] coordinates {
   (1, 4.3112e+03)
   (2, 6.5982e+03)
   (3, 7.4130e+03)
   (4, 7.6319e+03)
   (5, 8.7223e+03)
   (6, 8.7907e+03)
   (7, 9.4490e+03)
   (8, 9.8337e+03)}; \addlegendentry{CPU Time};
   \addplot+[] coordinates{(0,0)}; \addlegendentry{$\log_{10}\left(\|\nabla \times\vv_{\star}\|^2 - \|\nabla \times\vv\|^2\right)$}; 
  \end{axis}
  \begin{axis}[%
  xmode=normal,
  ymode=normal,
  xmin=1,xmax=8,
  legend style={at={(0.01,0.01)},anchor=south west},
  ylabel=$\log_{10}\left(\|\nabla \times\vv_{\star}\|^2 - \|\nabla \times\vv\|^2\right)$, axis y line*=right, y label style={at={(1.0,1.0)},anchor=south east,rotate=-90,color=blue},every y tick label/.style={blue},
  axis x line=none,
  cycle list name=stokescontrol,
  y filter/.code={\pgfmathparse{log10(\pgfmathresult)}\pgfmathresult},
  ]
  \pgfplotsset{cycle list shift=1};
  \addplot+[] coordinates {
   (1,    1.3570e-02)
   (2,    8.4398e-04)
   (3,    1.2053e-04)
   (4,    2.6202e-05)
   (5,    7.0580e-06)
   (6,    2.0170e-06)
   (7,    4.6300e-07)};
  \end{axis}
\end{tikzpicture}
\begin{tikzpicture}
  \begin{axis}[%
  xmode=normal,
  ymode=normal,
  xmin=1,xmax=8,
  legend style={at={(0.50,0.01)},anchor=south},
  ylabel=max. TT rank, axis y line*=left, y label style={at={(-0.01,1.0)},anchor=south west,rotate=-90},
  xlabel={$m$},
  cycle list name=stokescontrol
  ]
  \addplot+[] coordinates {
   (1, 24)
   (2, 36)
   (3, 46)
   (4, 54)
   (5, 59)
   (6, 64)
   (7, 67)
   (8, 71)};
   \pgfplotsset{cycle list shift=1}; \addlegendentry{max. TT rank};
   \addplot+[] coordinates{(0,0)}; \addlegendentry{$\log_{10}\mbox{MemR}$}; 
  \end{axis}
  \begin{axis}[%
  xmode=normal,
  ymode=normal,
  legend style={at={(0.01,0.01)},anchor=south west},
  ylabel=$\log_{10}\mbox{MemR}$, axis y line*=right, y label style={at={(1.0,1.0)},anchor=south east,rotate=-90,color=red},every y tick label/.style={red},
  axis x line=none,
  cycle list name=stokescontrol,
  y filter/.code={\pgfmathparse{log10(\pgfmathresult)}\pgfmathresult},
  ]
  \pgfplotsset{cycle list shift=2};
  \addplot+[] coordinates {
   (1, 8.7738e-04)
   (2, 1.4039e-04)
   (3, 2.1397e-05)
   (4, 2.9180e-06)
   (5, 3.9534e-07)
   (6, 5.3436e-08)
   (7, 7.6472e-09)
   (8, 1.0758e-09)};
  \end{axis}
\end{tikzpicture}
}
\end{figure}
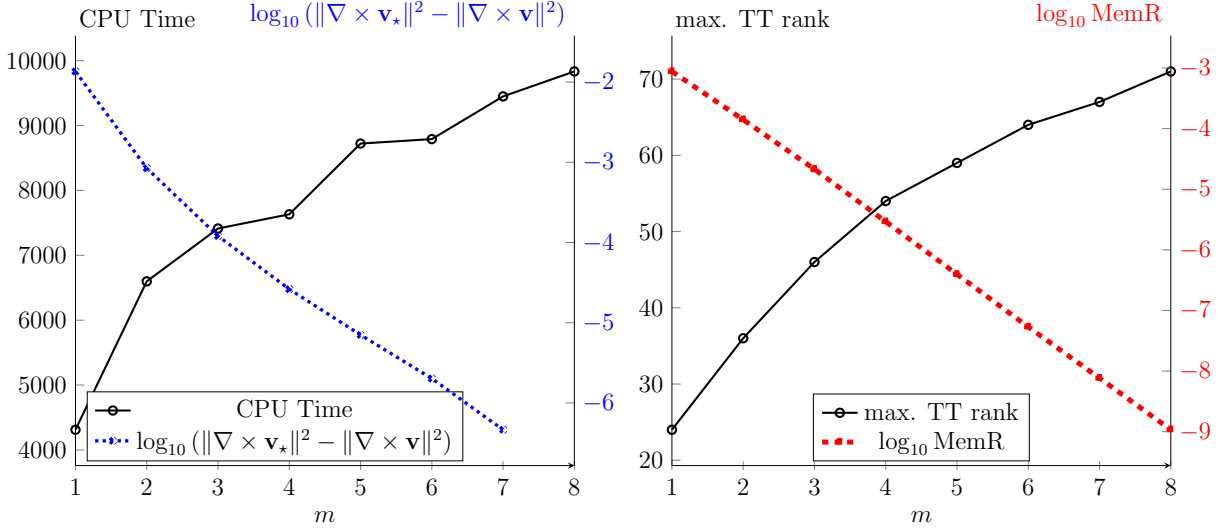

\subsubsection{Influence of the regularization}

In Fig. \ref{fig:beta}, we investigate how the model and the scheme depend on the regularization parameter $\beta$.
Setting $\beta \gg 1$ means solving almost an uncontrolled problem, while $\beta<1$ corresponds to a control with a certain ``power''.
We see that the CPU time and TT ranks grow significantly with $\beta$, when the model switches from effectively Stokes to an essentially nonlinear Navier-Stokes regime.
This is also reflected by a larger number of Picard iterations, needed for the system to converge.
The norm of vorticity starts growing fast when $\beta$ exceeds $1$.
On the other hand, for very large $\beta>10$ we observe a slight decrease of the CPU time due to a smaller number of the local GMRES iterations: the KKT matrix becomes anti-diagonally dominant, and the preconditioner \eqref{eq:sc2} becomes more efficient.

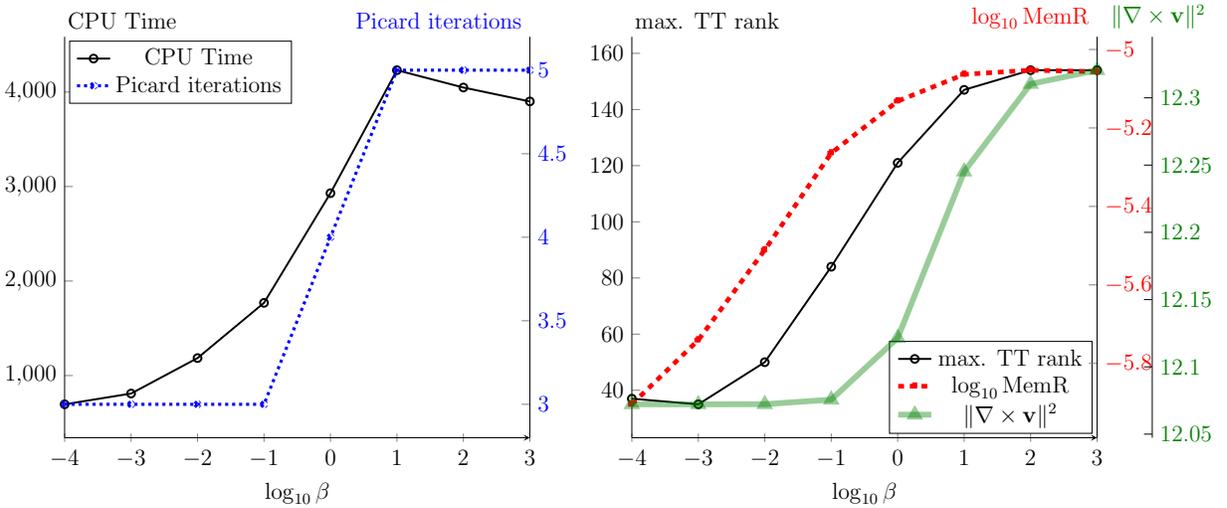
\begin{figure}[h!]
\caption{CPU time and the number of Picard iterations (left) and TT rank, memory reduction ratio and vorticity (right) for different regularization parameters.}
\label{fig:beta}
\resizebox{0.98\linewidth}{!}{
\begin{tikzpicture}
  \begin{axis}[%
  xmode=normal,
  ymode=normal,
  xmin=-4,xmax=3,
  xlabel={$\log_{10}\beta$},
  legend style={at={(0.01,0.99)},anchor=north west},
  ylabel=CPU Time, axis y line*=left, y label style={at={(-0.01,1.0)},anchor=south west,rotate=-90,color=black},
  cycle list name=stokescontrol
  ]

  \addplot+[] coordinates {
   (-4, 695.6047   )
   (-3, 809.5868   )
   (-2, 1.1849e+03 )
   (-1, 1.7693e+03 )
   (0,  2.9298e+03)
   (1,  4.2312e+03)
   (2,  4.0478e+03)
   (3,  3.9007e+03)}; \addlegendentry{CPU Time};
   \addplot+[] coordinates{(5,0)}; \addlegendentry{Picard iterations}; 
  \end{axis}
  \begin{axis}[%
  xmode=normal,
  ymode=normal,
  legend style={at={(0.01,0.01)},anchor=south west},
  ylabel=Picard iterations, axis y line*=right, y label style={at={(1.0,1.0)},anchor=south east,rotate=-90,color=blue},every y tick label/.style={blue},
  axis x line=none,
  cycle list name=stokescontrol
  ]
  \pgfplotsset{cycle list shift=1};
  \addplot+[] coordinates {
   (-4, 3 )
   (-3, 3 )
   (-2, 3 )
   (-1, 3 )
   (0,  4)
   (1,  5)
   (2,  5)
   (3,  5)};
  \end{axis}
\end{tikzpicture}
\begin{tikzpicture}
  \begin{axis}[%
  xmode=normal,
  ymode=normal,
  xmin=-4,xmax=3,
  legend style={at={(0.99,0.01)},anchor=south east},
  ylabel=max. TT rank, axis y line*=left, y label style={at={(-0.01,1.0)},anchor=south west,rotate=-90},
  xlabel={$\log_{10}\beta$},
  cycle list name=stokescontrol
  ]
  \addplot+[] coordinates {
   (-4, 37  )
   (-3, 35  )
   (-2, 50  )
   (-1, 84  )
   (0,  121)
   (1,  147)
   (2,  154)
   (3,  154)}; \addlegendentry{max. TT rank};
   \pgfplotsset{cycle list shift=1};
   \addplot+[] coordinates{(5,0)}; \addlegendentry{$\log_{10}\mbox{MemR}$}; 
   \addplot+[] coordinates{(5,0)}; \addlegendentry{$\|\nabla \times \vv\|^2$}; 
  \end{axis}
  \begin{axis}[%
  xmode=normal,
  ymode=normal,
  legend style={at={(0.01,0.01)},anchor=south west},
  ylabel=$\log_{10}\mbox{MemR}$, axis y line*=right, y label style={at={(1.0,1.0)},anchor=south east,rotate=-90,color=red},every y tick label/.style={red},
  axis x line=none,
  cycle list name=stokescontrol,
  y filter/.code={\pgfmathparse{log10(\pgfmathresult)}\pgfmathresult},
  ]
  \pgfplotsset{cycle list shift=2};
  \addplot+[] coordinates {
   (-4, 1.2465e-06 )
   (-3, 1.8179e-06 )
   (-2, 3.0838e-06 )
   (-1, 5.4453e-06 )
   (0,  7.3846e-06)
   (1,  8.6410e-06)
   (2,  8.8555e-06)
   (3,  8.7742e-06)};
  \end{axis}
  \begin{axis}[%
  xmode=normal,
  ymode=normal,
  legend style={at={(0.01,0.01)},anchor=south west},
  ylabel=$\|\nabla \times \vv\|^2$, axis y line*=right, y label style={at={(1.25,1.0)},anchor=south east,rotate=-90,color=green!50!black},
  axis x line=none,
  cycle list name=stokescontrol
  ]
  \pgfplotsset{cycle list shift=3};
  \pgfplotsset{every outer y axis line/.style={xshift=1cm}, every tick/.style={xshift=1cm}, every y tick label/.style={xshift=1cm,color=green!50!black} }
  \addplot+[] coordinates {
   (-4, 12.072155180 )
   (-3, 12.072156141 )
   (-2, 12.072237403 )
   (-1, 12.075644954 )
   (0,  12.121485198)
   (1,  12.245154757)
   (2,  12.310294779)
   (3,  12.320554819)};
  \end{axis}
\end{tikzpicture}
}
\end{figure}

Figures \ref{unsteady_Y_time_re_50} and \ref{unsteady_U_time_re_50} show the first two moments of the velocity and control for the final time.
For comparison, the uncontrolled flow is shown in Figure \ref{uncontrolled_Y_re_50}.
We can notice that the uncontrolled flow is more aligned to the top of the domain, while developing an eddy below the step.
The controlled minimal vorticity flow reflects the regime of a larger viscosity.
Moreover, the reattachment point shrinks significantly compared to that in Figure \ref{uncontrolled_Y_re_50}.

\begin{figure}[h!]
\centering
\caption{Plots of the mean (top) and variance (bottom) of the stream function for an unsteady distributed control flow with $\nu = 1/50$ at $t=10$}
\includegraphics[scale=0.6]{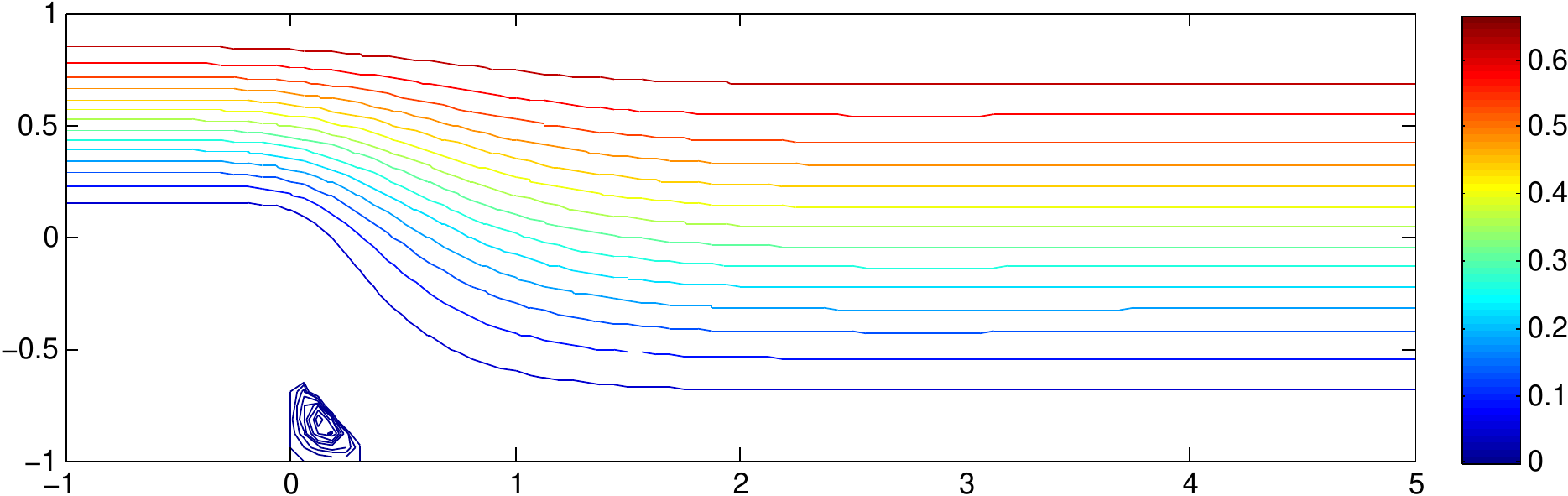}\\
\includegraphics[scale=0.6]{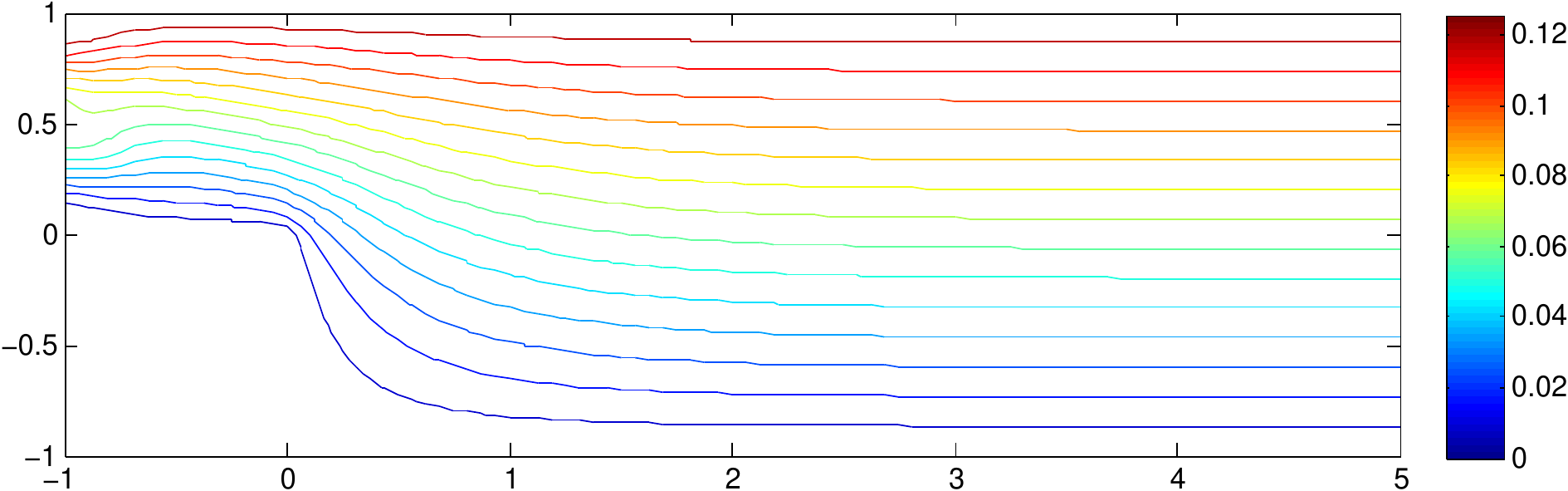}
\label{unsteady_Y_time_re_50}
 \vspace{-1.5mm}
\end{figure}

\begin{figure}[h!]
\centering
\caption{Plots of the mean (top) and variance (bottom) of the control for an unsteady distributed control flow with $\nu = 1/50,\; n_t=1024.$}
 \includegraphics[scale=0.8]{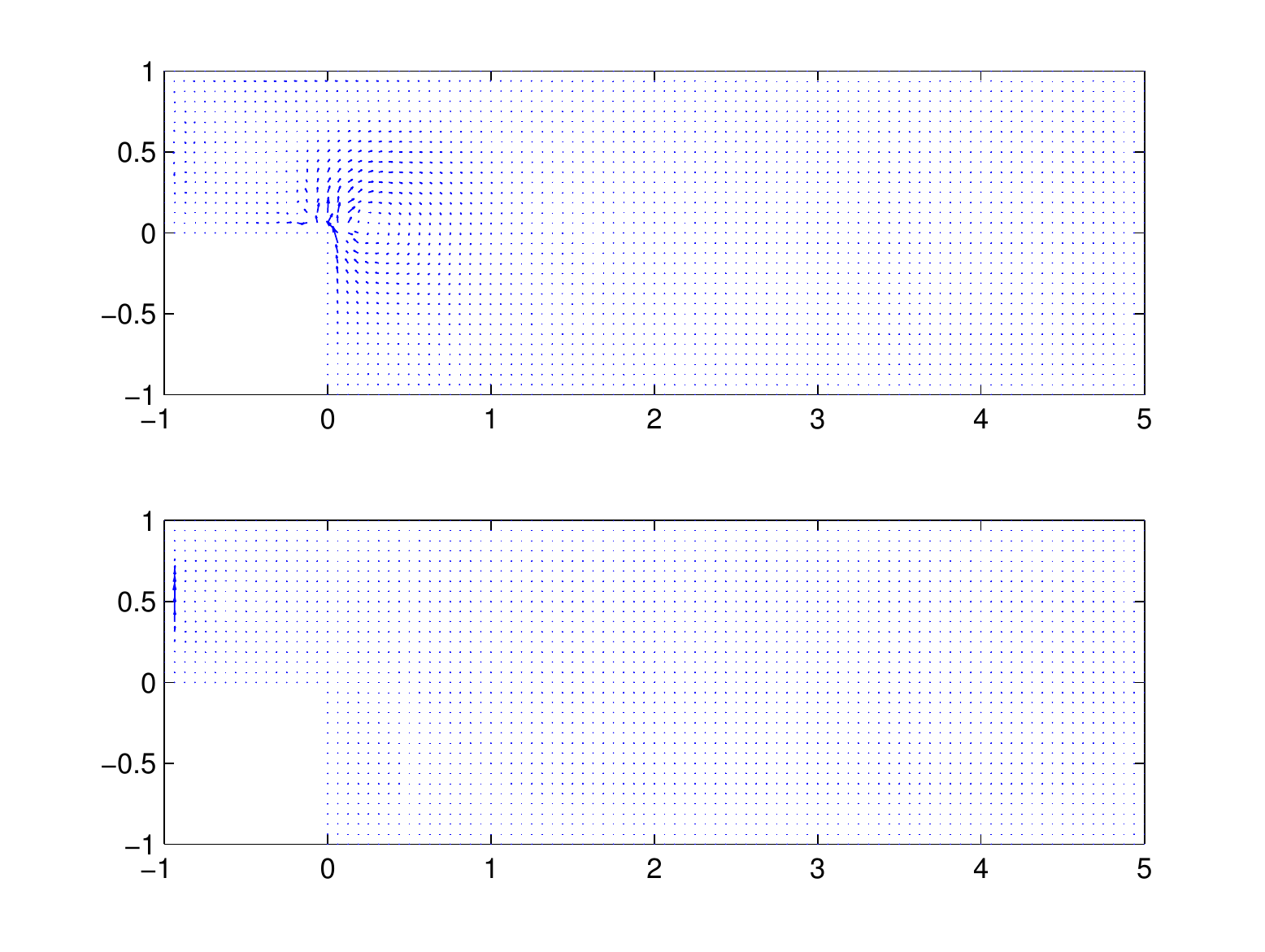}
\label{unsteady_U_time_re_50}
 \vspace{-1.5mm}
\end{figure}


\begin{figure}[h!]
\centering
\caption{Stream function for an unsteady uncontrolled flow at $t=10$.}
\includegraphics[scale=0.6]{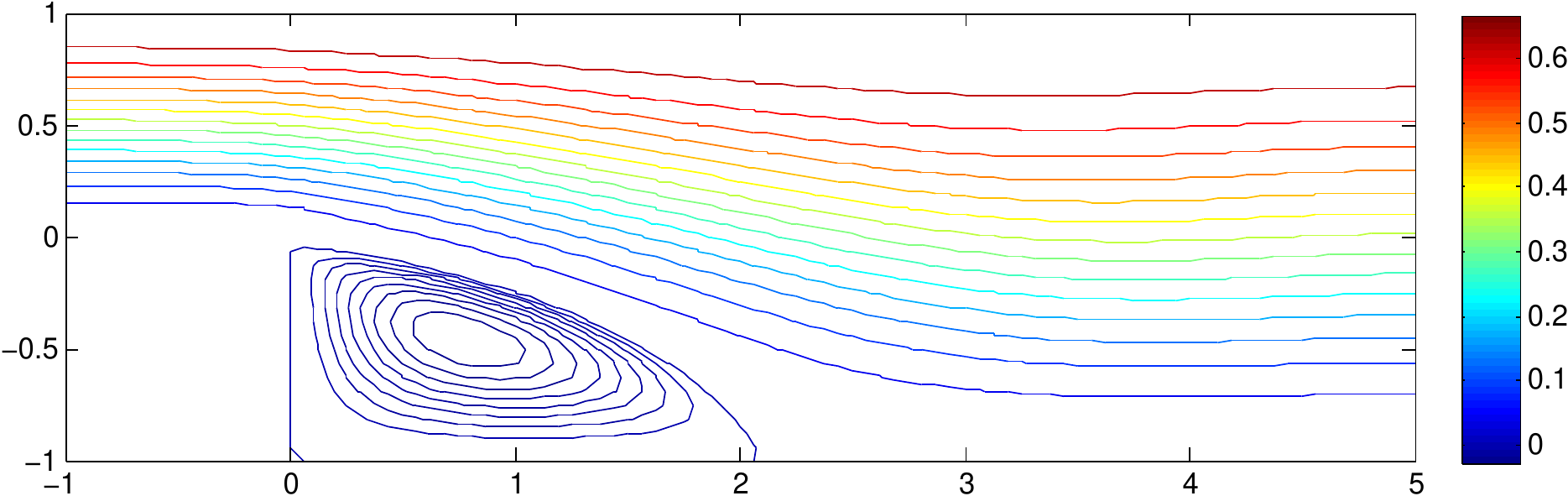}
\label{uncontrolled_Y_re_50}
 \vspace{-1.5mm}
\end{figure}


\subsection{Time-dependent boundary control problem}
Now we switch to the boundary control, applied as a Neumann boundary condition on the step wall \eqref{kkt_cond1}.
The problem is more difficult since the partial control cannot fully steer the flow to the Stokes regime.
Strongly nonlinear effects that remain in the flow inflate the TT ranks and slow down the convergence.

In particular, in Fig. \ref{fig:nu_bnd} we show how the scheme behaves for different viscosities.
We see that, compared to the distributed control case, even the case $\nu=1/50$ is difficult to solve due to very large TT ranks.
On the other hand, changing the regularization parameter (Fig. \ref{fig:beta_bnd}) does not influence the cost and storage complexities as much as in the distributed control case.
In fact, taking small $\beta$ in the partial control case leads to a poor performance of the Schur complement preconditioner,
and hence a large CPU time due to a large number of local GMRES iterations.
Nevertheless, smaller vorticity norm and number of Picard iterations for $\beta<1$ indicates that the target functional is optimized as expected.

%
%

\begin{figure}[h!]
\caption{CPU time and the number of Picard iterations (left) and TT rank and memory reduction ratio (right) for different viscosities with the boundary control.}
\label{fig:nu_bnd}
\resizebox{0.98\linewidth}{!}{
\begin{tikzpicture}
  \begin{axis}[%
  xmode=normal,
  ymode=normal,
  xmin=10,xmax=50,
  xlabel={$1/\nu$},
  ytick={00,10,20,30,40,50,60},
  yticklabels={0,10000,20000,30000,40000,50000,60000},
  ymin=0,ymax=65,
  y filter/.code={\pgfmathparse{0.001*\pgfmathresult}\pgfmathresult},
  legend style={at={(0.01,0.99)},anchor=north west},
  ylabel=CPU Time, axis y line*=left, y label style={at={(-0.01,1.0)},anchor=south west,rotate=-90,color=black},
  cycle list name=stokescontrol,
  ]
  \addplot+[] coordinates {
   (10,      4832.0  )
   (20,     10075.0  )
   (40,     37246.0  )
   (50,     62334.0  )}; \addlegendentry{CPU time};
   \addplot+[] coordinates{(0,0)}; \addlegendentry{Picard iterations};
  \end{axis}
  \begin{axis}[%
  xmode=normal,
  ymode=normal,
  legend style={at={(0.01,0.01)},anchor=south west},
  ylabel=Picard iterations, axis y line*=right, y label style={at={(1.0,1.0)},anchor=south east,rotate=-90,color=blue},every y tick label/.style={blue},
  axis x line=none,
  cycle list name=stokescontrol
  ]
  \pgfplotsset{cycle list shift=1};
  \addplot+[] coordinates {
   (10,     5 )
   (20,     7 )
   (40,     10)
   (50,     11)};
  \end{axis}
\end{tikzpicture}
\begin{tikzpicture}
  \begin{axis}[%
  xmode=normal,
  ymode=normal,
  xmin=10,xmax=50,
  legend style={at={(0.99,0.01)},anchor=south east},
  ylabel=max. TT rank, axis y line*=left, y label style={at={(-0.01,1.0)},anchor=south west,rotate=-90},
  xlabel={$1/\nu$},
  cycle list name=stokescontrol
  ]
  \addplot+[] coordinates {
   (10,     86 )
   (20,     116)
   (40,     161)
   (50,     178)}; \addlegendentry{max. TT rank};
   \pgfplotsset{cycle list shift=1};
   \addplot+[] coordinates{(0,0)}; \addlegendentry{$\log_{10}\mbox{MemR}$};
  \end{axis}
  \begin{axis}[%
  xmode=normal,
  ymode=normal,
  legend style={at={(0.01,0.01)},anchor=south west},
  ylabel=$\log_{10}\mbox{MemR}$, axis y line*=right, y label style={at={(1.0,1.0)},anchor=south east,rotate=-90,color=red},every y tick label/.style={red},
  axis x line=none,
  cycle list name=stokescontrol
  ]
  \pgfplotsset{cycle list shift=2};
  \addplot+[] coordinates {
   (10,      -3.8539)
   (20,      -3.2366)
   (40,      -3.0362)
   (50,      -2.9586)};
  \end{axis}
\end{tikzpicture}
}
\end{figure}
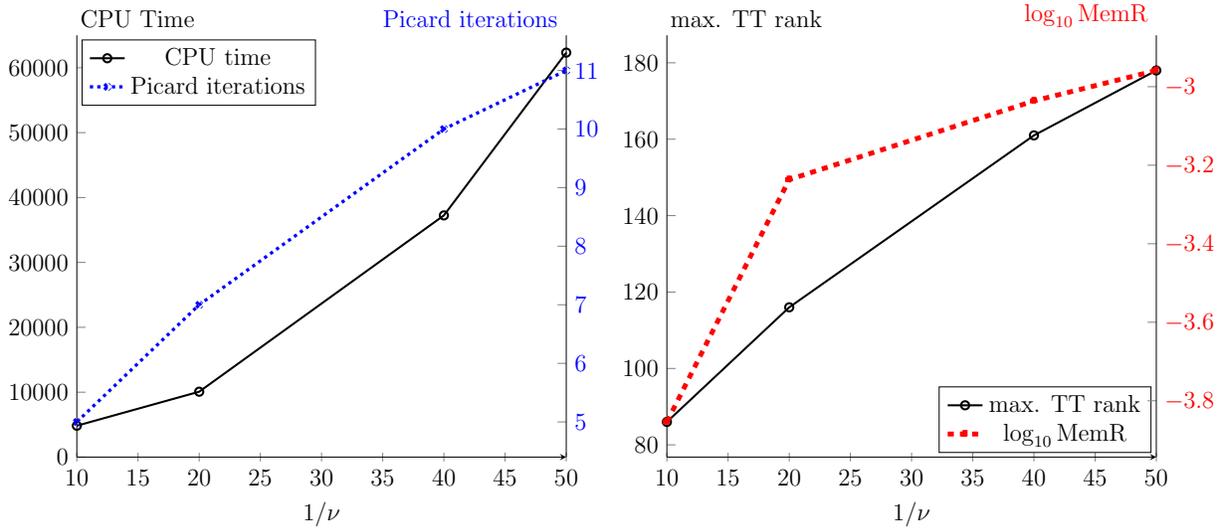

%
%

\begin{figure}[h!]
\caption{CPU time and the number of Picard iterations (left) and TT rank and vorticity (right) for different regularization parameters of the boundary control.}
\label{fig:beta_bnd}
\resizebox{0.98\linewidth}{!}{
\begin{tikzpicture}
  \begin{axis}[%
  xmode=normal,
  ymode=normal,
  xmin=-3,xmax=2,
  xlabel={$\log_{10}\beta$},
  ytick={10,12,14,16,18,20},
  yticklabels={10000,12000,14000,16000,18000,20000},
  ymin=9,ymax=21,
  y filter/.code={\pgfmathparse{0.001*\pgfmathresult}\pgfmathresult},
  legend style={at={(0.50,0.99)},anchor=north},
  ylabel=CPU Time, axis y line*=left, y label style={at={(-0.01,1.0)},anchor=south west,rotate=-90,color=black},
  cycle list name=stokescontrol,
  ]
  \addplot+[] coordinates {
   (-3,        19509.0 )
   (-2,        13342.0 )
   (-1,        09884.0 )
   (1 ,        10157.0 )
   (2 ,        11135.0 )}; \addlegendentry{CPU Time};
   \addplot+[] coordinates{(0,0)}; \addlegendentry{Picard iterations};
  \end{axis}
  \begin{axis}[%
  xmode=normal,
  ymode=normal,
  legend style={at={(0.01,0.01)},anchor=south west},
  ylabel=Picard iterations, axis y line*=right, y label style={at={(1.0,1.0)},anchor=south east,rotate=-90,color=blue},every y tick label/.style={blue},
  axis x line=none,
  cycle list name=stokescontrol
  ]
  \pgfplotsset{cycle list shift=1};
  \addplot+[] coordinates{
   (-3, 5)
   (-2, 5)
   (-1, 5)
   (1 , 6)
   (2 , 7)};
  \end{axis}
\end{tikzpicture}
\begin{tikzpicture}
  \begin{axis}[%
  xmode=normal,
  ymode=normal,
  xmin=-3,xmax=2,
  ymin=80,ymax=180,
  legend style={at={(0.50,0.99)},anchor=north},
  ylabel=max. TT rank, axis y line*=left, y label style={at={(-0.01,1.0)},anchor=south west,rotate=-90},
  xlabel={$\log_{10}\beta$},
  cycle list name=stokescontrol
  ]
  \addplot+[] coordinates{
   (-3,    155)
   (-2,    160)
   (-1,    154)
   (1 ,    161)
   (2 ,    160)}; \addlegendentry{max. TT rank};
   \pgfplotsset{cycle list shift=2};
   \addplot+[] coordinates{(0,0)}; \addlegendentry{$\|\nabla \times \vv\|^2$};
  \end{axis}
  \begin{axis}[%
  xmode=normal,
  ymode=normal,
  legend style={at={(0.01,0.01)},anchor=south west},
  ylabel=$\|\nabla \times \vv\|^2$, axis y line*=right, y label style={at={(1.0,1.0)},anchor=south east,rotate=-90,color=green!50!black},every y tick label/.style={green!50!black},
  axis x line=none,
  cycle list name=stokescontrol
  ]
  \pgfplotsset{cycle list shift=3};
  \addplot+[] coordinates{
   (-3, 11.1100 )
   (-2, 11.1100 )
   (-1, 11.1200 )
   (1 , 11.3300 )
   (2 , 11.4100 )};
  \end{axis}
\end{tikzpicture}
}
\end{figure}
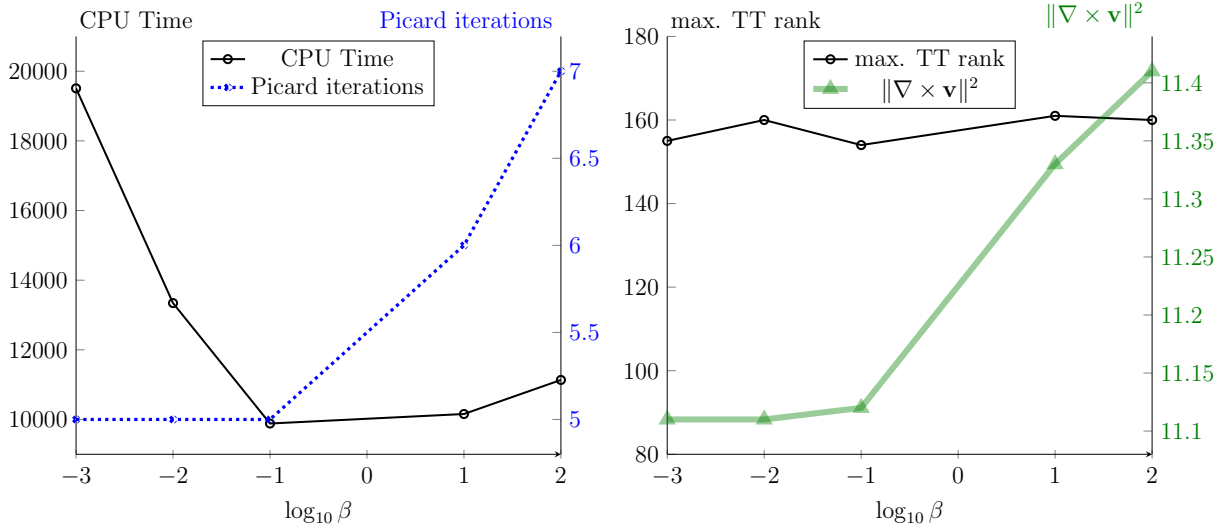

Concerning the stochastic parameters, we see qualitatively the same behaviour as
in the distributed control case: all complexity indicators increase for smaller $\gamma$ and larger $m$,
although the CPU time and TT ranks saturate with $m$, see Fig. \ref{fig:gamma_m_bnd}.

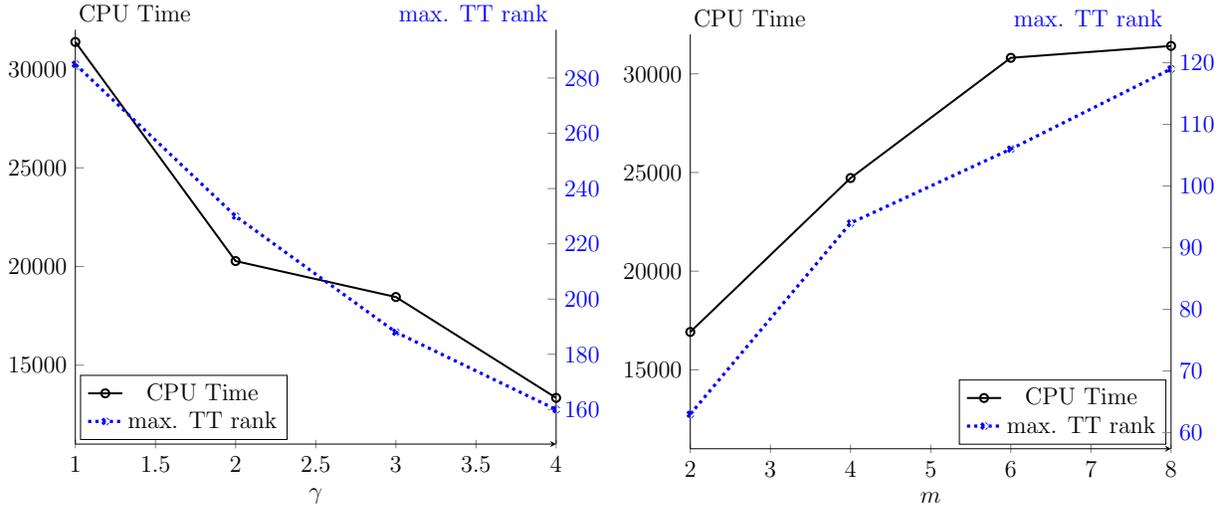
\begin{figure}[h!]
\caption{CPU time and the TT rank for different KLE decays (left) and KLE dimensions (right), boundary control.}
\label{fig:gamma_m_bnd}
\resizebox{0.98\linewidth}{!}{
\begin{tikzpicture}
  \begin{axis}[%
  xmode=normal,
  ymode=normal,
  xmin=1,xmax=4,
  xlabel={$\gamma$},
  ytick={15,20,25,30},
  yticklabels={15000,20000,25000,30000},
  ymin=11,ymax=32,
  y filter/.code={\pgfmathparse{0.001*\pgfmathresult}\pgfmathresult},
  legend style={at={(0.01,0.01)},anchor=south west},
  ylabel=CPU Time, axis y line*=left, y label style={at={(-0.01,1.0)},anchor=south west,rotate=-90,color=black},
  cycle list name=stokescontrol,
  ]
  \addplot+[] coordinates {
   (1, 31380.0)
   (2, 20274.0)
   (3, 18453.0)
   (4, 13342.0)}; \addlegendentry{CPU Time};
   \addplot+[] coordinates{(0,0)}; \addlegendentry{max. TT rank};
  \end{axis}
  \begin{axis}[%
  xmode=normal,
  ymode=normal,
  legend style={at={(0.01,0.01)},anchor=south west},
  ylabel=max. TT rank, axis y line*=right, y label style={at={(1.0,1.0)},anchor=south east,rotate=-90,color=blue},every y tick label/.style={blue},
  axis x line=none,
  cycle list name=stokescontrol
  ]
  \pgfplotsset{cycle list shift=1};
  \addplot+[] coordinates{
   (1, 285)
   (2, 230)
   (3, 188)
   (4, 160)};
  \end{axis}
\end{tikzpicture}
\begin{tikzpicture}
  \begin{axis}[%
  xmode=normal,
  ymode=normal,
  xmin=2,xmax=8,
  xlabel={$m$},
  ytick={15,20,25,30},
  yticklabels={15000,20000,25000,30000},
  ymin=11,ymax=32,
  y filter/.code={\pgfmathparse{0.001*\pgfmathresult}\pgfmathresult},
  legend style={at={(0.99,0.01)},anchor=south east},
  ylabel=CPU Time, axis y line*=left, y label style={at={(-0.01,1.0)},anchor=south west,rotate=-90,color=black},
  cycle list name=stokescontrol,
  ]
  \addplot+[] coordinates {
   (2, 16918.0)
   (4, 24720.0)
   (6, 30813.0)
   (8, 31414.0)}; \addlegendentry{CPU Time};
   \addplot+[] coordinates{(0,0)}; \addlegendentry{max. TT rank};
  \end{axis}
  \begin{axis}[%
  xmode=normal,
  ymode=normal,
  legend style={at={(0.01,0.01)},anchor=south west},
  ylabel=max. TT rank, axis y line*=right, y label style={at={(1.0,1.0)},anchor=south east,rotate=-90,color=blue},every y tick label/.style={blue},
  axis x line=none,
  cycle list name=stokescontrol
  ]
  \pgfplotsset{cycle list shift=1};
  \addplot+[] coordinates{
   (2, 63 )
   (4, 94 )
   (6, 106)
   (8, 119)};
  \end{axis}
\end{tikzpicture}
}
\end{figure}

\subsection{Stationary boundary control problem}
One of the sources of the rank inflation in the previous section is the simultaneous storage of
all time snapshots in a single TT representation.
In the next tests we consider the stationary Navier-Stokes equations as constraints.
In Fig. \ref{fig:nu_stat}, we vary the viscosity, and in Fig. \ref{fig:beta_stat} we investigate the influence of the regularization parameter.
We see that the results reflect qualitatively the behaviour of the time-dependent problem with the
boundary control, i.e. the stochastic components are also strongly coupled due to the nonlinearity, 
and it needs more Picard iterations and larger TT ranks for the solution to converge in a low viscosity regime.
Nonetheless, the TT ranks (and hence the computational times) are much smaller than in the time-dependent case.

%


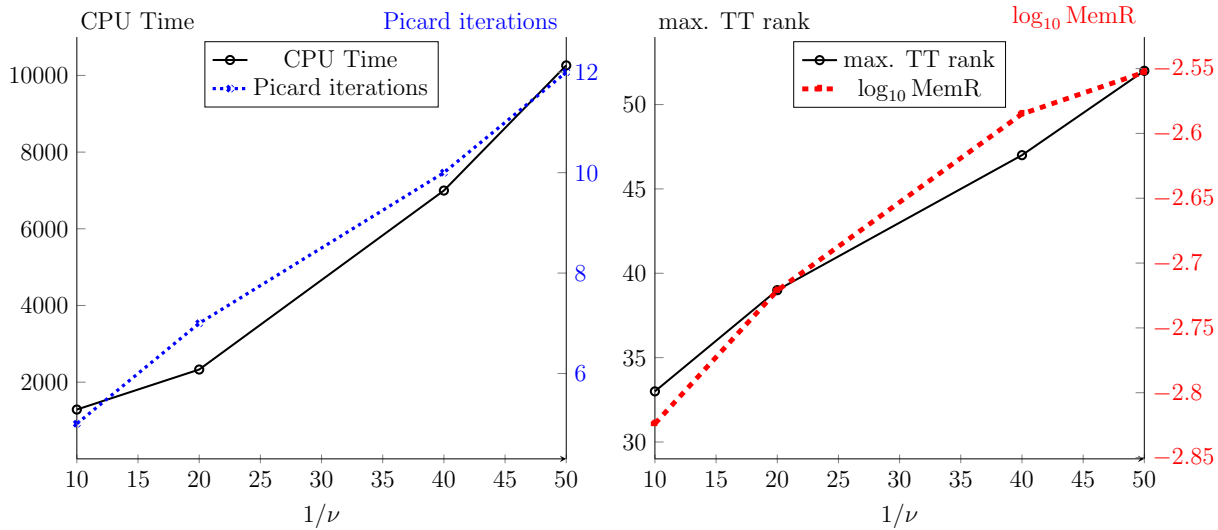
\begin{figure}[h!]
\caption{CPU time and number of Picard iterations (left) and TT rank and memory reduction ratio (right) for different viscosities, boundary control, stationary problem.}
\label{fig:nu_stat}
\resizebox{0.98\linewidth}{!}{
\begin{tikzpicture}
  \begin{axis}[%
  xmode=normal,
  ymode=normal,
  xmin=10,xmax=50,
  xlabel={$1/\nu$},
  ytick={2,4,6,8,10},
  yticklabels={2000,4000,6000,8000,10000},
  ymin=0,ymax=11,
  y filter/.code={\pgfmathparse{0.001*\pgfmathresult}\pgfmathresult},
  legend style={at={(0.50,0.99)},anchor=north},
  ylabel=CPU Time, axis y line*=left, y label style={at={(-0.01,1.0)},anchor=south west,rotate=-90,color=black},
  cycle list name=stokescontrol,
  ]
  \addplot+[] coordinates {
   (10,    1286.10)
   (20,    2329.90)
   (40,    6997.20 )
   (50,   10261.00 )}; \addlegendentry{CPU Time};
  \addplot+[] coordinates{(0,0)}; \addlegendentry{Picard iterations};
  \end{axis}
  \begin{axis}[%
  xmode=normal,
  ymode=normal,
  legend style={at={(0.01,0.01)},anchor=south west},
  ylabel=Picard iterations, axis y line*=right, y label style={at={(1.0,1.0)},anchor=south east,rotate=-90,color=blue},every y tick label/.style={blue},
  axis x line=none,
  cycle list name=stokescontrol
  ]
  \pgfplotsset{cycle list shift=1};
  \addplot+[] coordinates {
   (10,      5)
   (20,      7)
   (40,     10)
   (50,     12)};
  \end{axis}
\end{tikzpicture}
\begin{tikzpicture}
  \begin{axis}[%
  xmode=normal,
  ymode=normal,
  xmin=10,xmax=50,
  ymin=29,ymax=54,
  legend style={at={(0.50,0.99)},anchor=north},
  ylabel=max. TT rank, axis y line*=left, y label style={at={(-0.01,1.0)},anchor=south west,rotate=-90},
  xlabel={$1/\nu$},
  cycle list name=stokescontrol
  ]
  \addplot+[] coordinates {
   (10,     33)
   (20,     39)
   (40,     47)
   (50,     52)}; \addlegendentry{max. TT rank};
  \pgfplotsset{cycle list shift=1};
  \addplot+[] coordinates{(0,0)}; \addlegendentry{$\log_{10}\mbox{MemR}$};
  \end{axis}
  \begin{axis}[%
  xmode=normal,
  ymode=normal,
  legend style={at={(0.01,0.01)},anchor=south west},
  ylabel=$\log_{10}\mbox{MemR}$, axis y line*=right, y label style={at={(1.0,1.0)},anchor=south east,rotate=-90,color=red},every y tick label/.style={red},
  axis x line=none,
  cycle list name=stokescontrol
  ]
  \pgfplotsset{cycle list shift=2};
  \addplot+[] coordinates {
   (10,   -2.8239)
   (20,   -2.7212)
   (40,   -2.5850)
   (50,   -2.5528)};
  \end{axis}
\end{tikzpicture}
}
\end{figure}
%
%
%

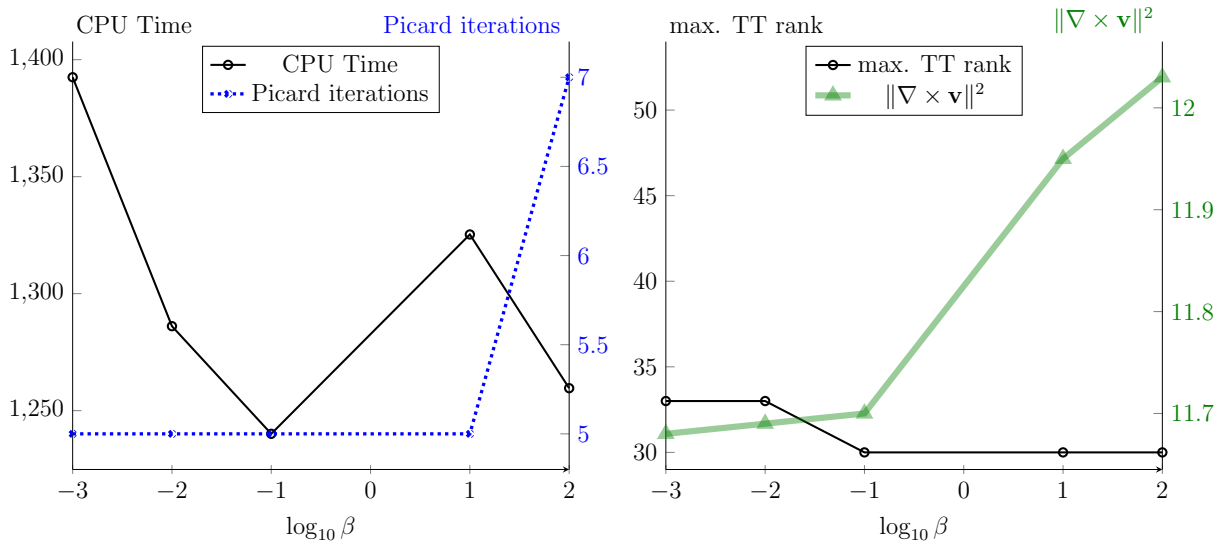
\begin{figure}[h!]
\caption{CPU time and number of Picard iterations (left) and TT rank and vorticity (right) for different regularization parameters, boundary control, stationary problem.}
\label{fig:beta_stat}
\resizebox{0.98\linewidth}{!}{
\begin{tikzpicture}
  \begin{axis}[%
  xmode=normal,
  ymode=normal,
  xmin=-3,xmax=2,
  xlabel={$\log_{10}\beta$},
  legend style={at={(0.50,0.99)},anchor=north},
  ylabel=CPU Time, axis y line*=left, y label style={at={(-0.01,1.0)},anchor=south west,rotate=-90,color=black},
  cycle list name=stokescontrol,
  ]
  \addplot+[] coordinates {
   (-3,    1392.5 )
   (-2,    1286.1 )
   (-1,    1240.1 )
   (1 ,    1325.3 )
   (2 ,    1259.6 )}; \addlegendentry{CPU Time};
  \addplot+[] coordinates{(5,0)}; \addlegendentry{Picard iterations};
  \end{axis}
  \begin{axis}[%
  xmode=normal,
  ymode=normal,
  legend style={at={(0.01,0.01)},anchor=south west},
  ylabel=Picard iterations, axis y line*=right, y label style={at={(1.0,1.0)},anchor=south east,rotate=-90,color=blue},every y tick label/.style={blue},
  axis x line=none,
  cycle list name=stokescontrol
  ]
  \pgfplotsset{cycle list shift=1};
  \addplot+[] coordinates{
   (-3, 5)
   (-2, 5)
   (-1, 5)
   (1 , 5)
   (2 , 7)};
  \end{axis}
\end{tikzpicture}
\begin{tikzpicture}
  \begin{axis}[%
  xmode=normal,
  ymode=normal,
  xmin=-3,xmax=2,
  ymin=29,ymax=54,
  legend style={at={(0.50,0.99)},anchor=north},
  ylabel=max. TT rank, axis y line*=left, y label style={at={(-0.01,1.0)},anchor=south west,rotate=-90},
  xlabel={$\log_{10}\beta$},
  cycle list name=stokescontrol
  ]
  \addplot+[] coordinates{
   (-3,    33)
   (-2,    33)
   (-1,    30)
   (1 ,    30)
   (2 ,    30)}; \addlegendentry{max. TT rank};
  \pgfplotsset{cycle list shift=2};
  \addplot+[] coordinates{(0,0)}; \addlegendentry{$\|\nabla \times \vv\|^2$};
  \end{axis}
  \begin{axis}[%
  xmode=normal,
  ymode=normal,
  legend style={at={(0.01,0.01)},anchor=south west},
  ylabel=$\|\nabla \times \vv\|^2$, axis y line*=right, y label style={at={(1.0,1.0)},anchor=south east,rotate=-90,color=green!50!black},every y tick label/.style={green!50!black},
  axis x line=none,
  cycle list name=stokescontrol
  ]
  \pgfplotsset{cycle list shift=3};
  \addplot+[] coordinates{
   (-3,    11.6800)
   (-2,    11.6900)
   (-1,    11.7000)
   (1 ,    11.9500)
   (2 ,    12.0300)};
  \end{axis}
\end{tikzpicture}
}
\end{figure}

Finally,  Figures \ref{stationary_partial_Y_re_50} and \ref{stationary_partial_U_re_50},
depict the moments of the velocity and control of the stationary boundary controlled flow.
The boundary controlled flow still manifests an eddy around the corner.
Nonetheless, it corresponds to a smaller vorticity than in the uncontrolled
regime (Fig. \ref{uncontrolled_Y_re_50}), where there is a rather sharp 
interface between the main stream and the corner area.
We also remark that the number of Picard iterations remains generally below 10, with just a few exceptions.

\begin{figure}[h!]
\centering
\caption{Plots of the mean (top) and variance (bottom) of the stream function for a stationary boundary control   flow with $\nu = 1/50.$}
\includegraphics[scale=0.6]{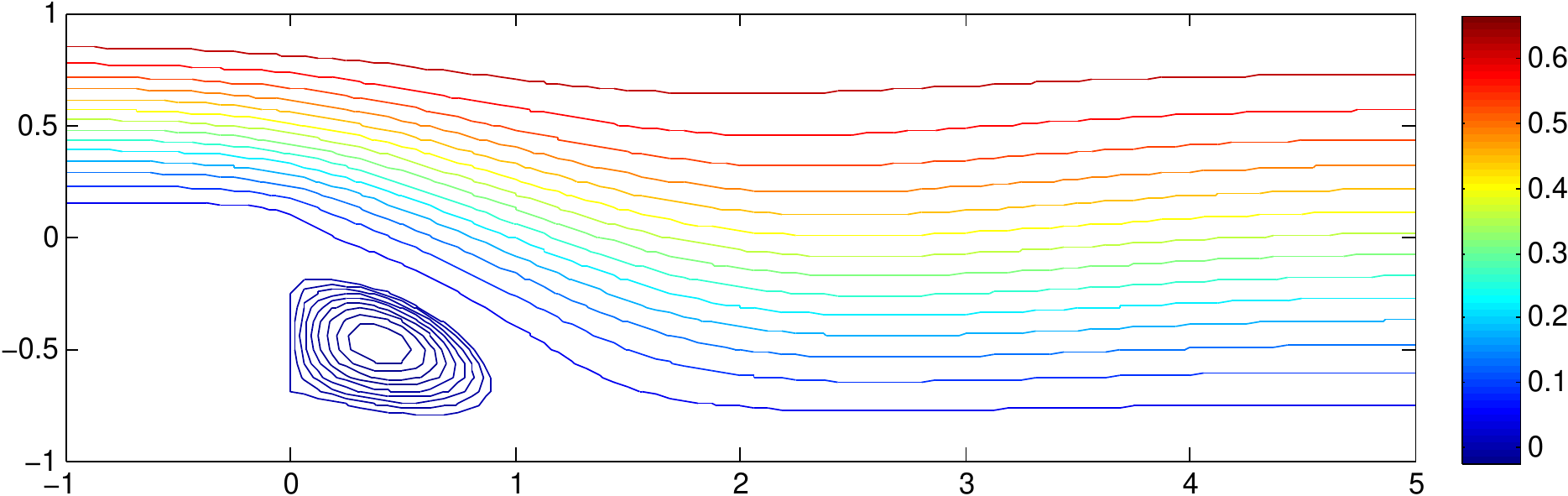} \\
\includegraphics[scale=0.6]{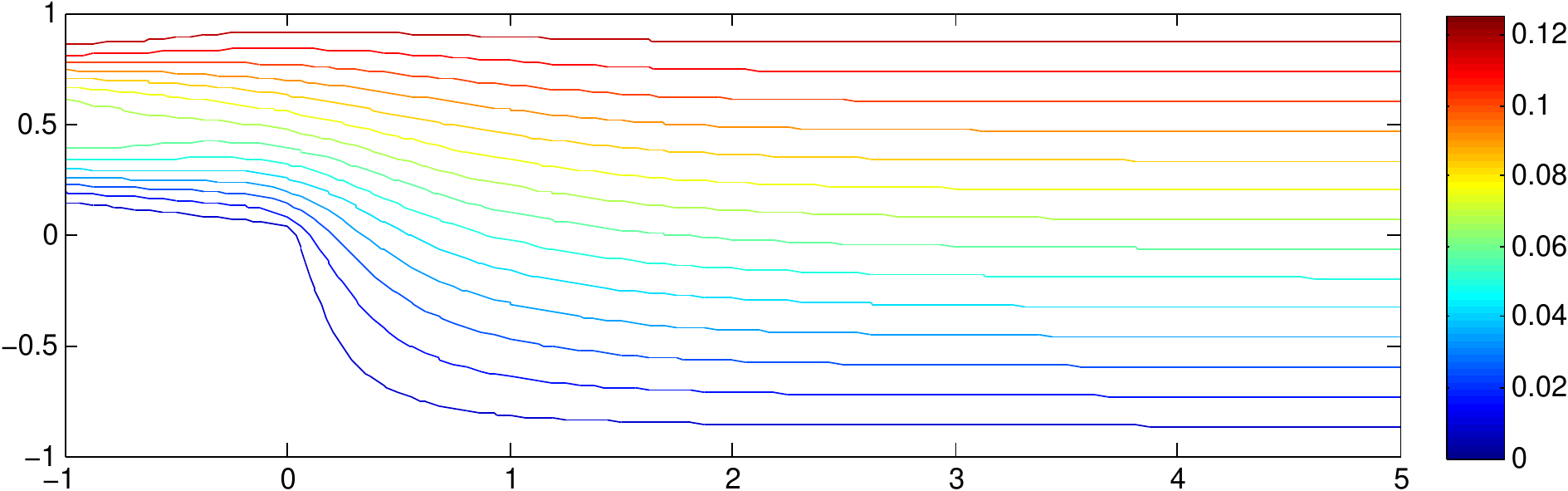}
\label{stationary_partial_Y_re_50}
 \vspace{-1.5mm}
\end{figure}

\begin{figure}[h!]
\centering
\caption{Plots of the mean (top) and variance (bottom) of control for a  stationary boundary control  flow with $\nu = 1/50.$}
 \includegraphics[scale=0.8]{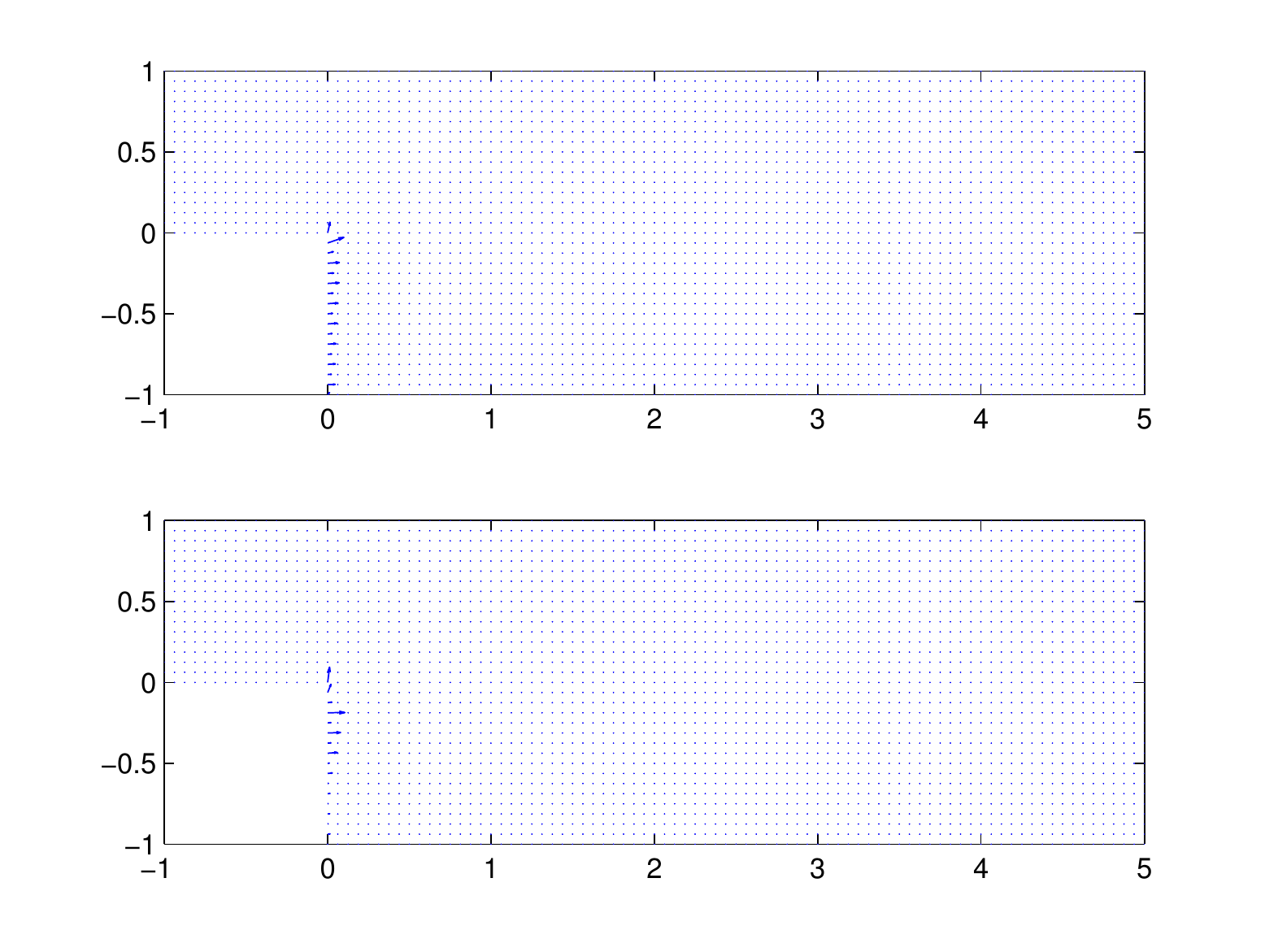}
\label{stationary_partial_U_re_50}
 \vspace{-1.5mm}
\end{figure}

\section{Conclusions and outlook}
\label{con_out}
We demonstrated the applicability of low-rank tensor decompositions to the
solution of optimal control problems constrained by unsteady Navier-Stokes equations with stochastic inputs.
This problem has a threefold challenge: a nonlinear time-dependent PDE, an optimization problem using a Lagrangian approach, and random inputs.
In a classical stochastic Galerkin discretization one needs to multiply the numbers of degrees of freedom coming from space, time and stochastic quantities.
In particular, the outer optimization implies storing all time snapshots, and random quantities are introduced as independent variables.

We never store or compute all elements of such a solution as this would quickly become infeasible even on large computers.
Instead we compress them into the Tensor Train representation.
A crucial part of the scheme is an alternating iterative solver,
which computes the TT factors directly.
Although  known in the multilinear algebra community, this idea required substantial modifications in order to be applicable to the OPNS.
We preserve the saddle-point structure in the reduced model
and accommodate components of different sizes, such as the boundary control.
Moreover, given a low-rank representation of the previous solution, we assemble the linearized operator also in a low-rank form.
The scheme provides a significant reduction of complexity, up to 9 orders of magnitude for the largest number of random variables.


\bibliographystyle{siam}
\bibliography{refs}
\end{document}